\documentclass[11pt,letterpaper]{article}
\usepackage{nsd, xspace, thmtools}
\usepackage{refcount}
\crefname{fact}{Fact}{Facts}
\Crefname{fact}{Fact}{Facts}

\newcommand{\Vladut}{Vl\u{a}du\c{t}\xspace}
\DeclareMathOperator{\CVP}{CVP}

\title{Lattices with Large Kissing Number from Local Density}

\author{Huck Bennett\thanks{University of Colorado Boulder. \email{huck.bennett@colorado.edu}. Supported in part by NSF Award No. 2432132.} \and 
Minki Hhan\thanks{Korea Advanced Institute of Science and Technology. \email{minkihhan@kaist.ac.kr}.}
\and
Zihan Li\thanks{National University of Singapore and Cornell University. \email{zihan\_li\_05@u.nus.edu}.}
\and Noah Stephens-Davidowitz\thanks{Cornell University. \email{noahsd@gmail.com}. Supported in part by NSF Grant No.~CCF-2312296, a Packard Foundation Fellowship, and a generous gift from Google.}}

\date{October 9, 2026}

\begin{document}

\maketitle

\begin{abstract}
Let $\kappa_p(n)$ denote the maximum kissing number of a lattice $\lat \subset \R^n$.  
We prove that $\kappa_p(n) \geq 2^{\Omega(\sqrt{n})}$ for all $p\ge 1$ and that $\kappa_p(n) \geq \Omega(\alpha_p^n)$ for all $p > 2$, where $\alpha_p > 1$ is a constant depending only on $p$. In particular, our former result improves on the previously best known lower bound of $\kappa_2(n) \geq 2^{\Theta(\log^2 n)}$ for the $\ell_2$ kissing number, which has stood for over $60$ years.
(In concurrent and independent work, Laarhoven and Kominers [arXiv, 2026] showed a similar lower bound on $\kappa_2(n)$ using similar techniques.)
In fact, our lower bounds follow from \emph{efficient} (randomized) constructions of lattices with large kissing number.

Our results use a general construction of lattices with large kissing number.
Specifically, we show a simple equivalence between such lattices and a new object that we call an \emph{exact locally dense lattice}.
Closely related objects appear throughout the literature on computational lattice problems.

Our constructions of lattices with large kissing number then follow as corollaries by showing that constructions due to Bennett and Peikert [RANDOM, 2023] (using Craig lattices) and Aggarwal and Stephens-Davidowitz [STOC, 2018] (using sparsification of $\Z^n$) in fact give exact locally dense lattices, as well as a new construction of exact locally dense lattices built from rotations of the $m$-dimensional checkerboard lattice $D_m$.
\end{abstract}

\mytableofcontents

\section{Introduction}

A lattice $\lat \subset \R^m$ is the set of all integer linear combinations of $n$ linearly independent vectors $\vec{b}_1, \ldots, \vec{b}_n$. That is, $\lat = \lat(\basis) := \{\sum_{i=1}^n a_i \vec{b}_i : a_1, \ldots, a_n \in \Z\}$, where $\basis := (\vec{b}_1, \ldots, \vec{b}_n)$ is a basis of $\lat$. The \emph{rank} $\rank(\lat)$ of $\lat$ is the number $n$ of vectors in a basis of $\lat$.

The minimum distance of a lattice in the $\ell_p$ norm is $\lambda_1^{(p)}(\lat) := \min_{\vec{x} \in \lat \setminus \{\vec{0}\}} \|\vec{x}\|_p$, where $\|\cdot\|_p$ is the $\ell_p$ norm.
A fundamental geometric quantity associated with a lattice $\lat$ is its \emph{kissing number} in the $\ell_p$ norm, which is defined as
\[
\kappa_p(\lat) := |{\{\vec{x} \in \lat : \|\vec{x}\|_p = \lambda_1^{(p)}(\lat)}\}| \ \text{.}
\]
That is, $\kappa_p(\lat)$ is the number of shortest non-zero vectors in $\lat$ in the $\ell_p$ norm.
We extend this notation and define the \emph{lattice kissing number}
\[
\kappa_p(n) := \max_{\lat \, : \, \rank(\lat) = n} \kappa_p(\lat)
\]
to be the maximum $\ell_p$ kissing number of a lattice $\lat \subset \R^m$ of rank $n$ (in any ambient dimension $m$).
The focus of this work is on giving an asymptotic lower bound on $\kappa_p(n)$ for different values of $p$.%
\footnote{The case when $p=\infty$ is already resolved. Specifically, $\kappa_\infty(n)=3^n-1$, and this is achieved by $\Z^n$.
To show that $\kappa_\infty(n) \leq 3^n-1$, we note that the shortest non-zero vectors of $\lat$ must lie in distinct, non-zero cosets modulo $3 \lat$. Indeed, suppose that vectors $\vec{x}, \vec{y} \in \lat$ are such that $\norm{\vec{x}}_{\infty} = \norm{\vec{y}}_{\infty} = \lambda_1^{(\infty)}(\lat)$ and that $\vec{x} = \vec{y} + 3 \vec{v}$ for $\vec{v} \in \lat \setminus \set{\vec{0}}$.
Then by triangle inequality, $\lambda_1^{(\infty)}(\lat) = \|\vec{x} \|_\infty \ge 3\|\vec v\|_\infty-\|\vec y\|_\infty\ge 2 \lambda_1^{(\infty)}(\lat)$, which is a contradiction.}

There is substantial interest both in the value of $\kappa_p(n)$ for fixed, small values of $n$ and in the asymptotic behavior of $\kappa_p(n)$, especially in the important Euclidean case of $p = 2$. See, e.g.,~\cite[Chapter 1, Section 2]{CS99} and~\cite{alonPackingsLargeMinimum1997}.
Prior to this work, the best known lower bound on $\kappa_2(n)$ was the quasipolynomial bound $\kappa_2(n) \geq 2^{(1/2 + o(1)) \cdot \log_2(n)^2} = 2^{\Theta(\log^2 n)}$ shown by Leech~\cite{Leech1964SpherePackings} for Barnes-Wall lattices~\cite{barnesExtremeFormsDefined1959}; see also~\cite[Chapter 1, Equation 50]{CS99}.
(A series of recent works~\cite{Vladut-l2kissnum-2019,Vladut-lpkissnum-2021,vladut24} claimed to show a family of lattices with exponentially large kissing number. 
Specifically, \cite{Vladut-l2kissnum-2019,vladut24} claimed to prove that $\kappa_2(n) \geq 2^{\Omega(n)}$ and \cite{Vladut-lpkissnum-2021} claimed to show that $\kappa_p(n) \geq 2^{\Omega(n)}$ for all $p$. Unfortunately, these works were found to be invalid in~\cite{journals/tit/BennettGS25} and were subsequently withdrawn.)

The analogous question of whether there exists a family of binary linear codes with exponential kissing number was asked in~\cite{kalaiDistanceDistributionCodes1995} and answered in the affirmative by Ashikhmin, Barg, and \Vladut in~\cite{Ashikhmin-Barg-Vladut-Exp-Kiss-Codes-01}.
Alon~\cite{alonPackingsLargeMinimum1997} gave a non-lattice packing of unit spheres in $\R^n$, where each sphere touches at least $2^{\sqrt{n}}$ others. 
We note also that Aggarwal and Stephens-Davidowitz~\cite{AS18} gave complexity-theoretic applications of lattices with exponential kissing number. They showed that, assuming a generic complexity-theoretic assumption (namely, non-uniform Gap-ETH), the existence of a family of lattices with exponential Euclidean kissing number implies there is no $2^{o(n)}$-time algorithm for the computational lattice problem $\SVP_p$ for any $1 \le p \le 2$.\footnote{$\SVP_p$ is the Shortest Vector Problem in the $\ell_p$ norm, i.e., the problem of computing a vector $\vec{v} \in \lat$ in a given lattice $\lat$ with $\norm{\vec{v}}_p = \lambda_1^{(p)}(\lat)$.}
Finally, Bennett, Peikert, and Tang~\cite{conf/innovations/BennettPT22} showed how exponential kissing number lattices can be transformed into locally dense lattice gadgets. In this paper, we are working in the reverse direction, i.e., converting suitable locally dense lattices into lattices with large kissing number. 
 (We make the close relationship between locally dense lattices with many close vectors and lattices with large kissing number precise in \cref{thm:LDL-KN-equivalence}.)

\subsection{Our contribution}
\label{sec:our_constribution}

In terms of $\kappa_p(n)$ itself, our main result is a new asymptotic lower bound on $\kappa_p(n)$ for all $\ell_p$. That is, for every $p \geq 1$ and every sufficiently large $n$, we construct a lattice $\lat$ of rank $n$ (and ambient dimension $n$) and prove that $\kappa_p(\lat)$ is (much) larger than what prior work was able to show.

\begin{theorem}[Informal; see \cref{cor:craig kn,thm:last-one}] \label{thm:kn-lower-bounds-main} \ 
    \begin{enumerate}
        \item \label{item:main_theorem_subexp}
    For all $p \geq 1$, there exists a lattice $\lat \subset \R^n$ with rank $n$ such that 
    \[
        \kappa_p(\lat) \geq e^{\sqrt{n} - o(\sqrt{n})}
        \; .
    \]
    \item \label{item:main_theorem_exp} For all $p > 2$ there exists a lattice $\lat \subset \R^n$ with rank $n$ such that
    \[
        \kappa_p(\lat) \geq \Omega(\alpha_p^n)
        \; ,
    \]
    for some (efficiently computable) constant $\alpha_p > 1$ depending only on $p$. Furthermore,  $\lim_{p \to \infty} \alpha_p = 2$.
    \end{enumerate}
\end{theorem}

In particular, the result in \cref{item:main_theorem_exp} that holds for $p > 2$ is tight up to the constant $\alpha_p$, since it is easy to show (e.g., via a packing bound) that $\kappa_p(n) \leq 2^{O(n)}$. So, we prove that the kissing number $\kappa_p(n)$ is $2^{\Theta(n)}$ for all $p > 2$.  In fact, $\kappa_p(n) \leq O(2^n)$ for all finite $p \geq 1$, so that as $p \to \infty$, \cref{item:main_theorem_exp} is quite tight.\footnote{See~\cite{millerKissingNumbersTransference2019,millerKissingNumberCrosspolytope2025} for the best known upper bounds on the constant in the exponent for $p \neq 2$. In particular, \cite{millerKissingNumberCrosspolytope2025} proves that $\kappa_1(n) \leq 12(2^n-1)$ and it is trivial to see that $\kappa_p(n) \leq 2 (2^n-1)$ for any finite $p > 1$. See \cite{kabatjanskiiBoundsPackingsSphere1978} for the best known upper bound for $p = 2$.} 

For $p=2$ (and more generally $1 \leq p \leq 2$), we do not know if our result is tight, but it is far stronger than the prior best lower bound of $\kappa_2(n) \geq 2^{\Omega(\log(n)^2)}$~\cite{Leech1964SpherePackings}. In independent work that appeared as a preprint before the present work was public, Laarhoven and Kominers proved essentially the same result for $p = 2$ (with a slightly better constant in the exponent) using essentially the same construction~\cite{laarhoven2026latticefamilykissingconstant}. (See \cref{sec:concurrent}.) We do not know of any prior lower bounds for $p \neq 2$ besides the rather simple lower bound of $2n(n-1)$ achieved by the lattice $D_n$.

We note that our constructions of lattices with large kissing number actually show a \emph{randomized} $\poly(n)$-time procedure that outputs a basis of a lattice $\lat \subset \R^n$ that has the desired kissing number with high probability. (In fact, a single lattice $\lat \subset \R^n$ achieves $\kappa_p(\lat) \geq e^{\sqrt{n} - o(\sqrt{n})}$ simultaneously for all $p \geq 1$ with high probability.) One might wish for an efficient deterministic procedure that outputs a basis for the lattice. As an additional contribution, we show (in \cref{cor:explicit-construction}) a lattice $\lat \subset \R^n$ with a very simple description (a Craig lattice with one explicit vector added to its basis) that achieves $\kappa_p(\lat) \geq e^{\Omega(n^{1/4}\log n)}$ for all $p \geq 1$. (Here, again, a single lattice works simultaneously for all $p$.)

Our proof of \cref{thm:kn-lower-bounds-main} uses a new careful definition of objects that we call \emph{exact locally dense lattices}. This definition is our second contribution. As the name suggests, exact locally dense lattices are slight variants of well-studied objects called \emph{locally dense lattices}, which are ubiquitous in the study of the complexity of computational problems on lattices (as we discuss below).

\begin{definition}[Exact locally dense lattice] \label{def:locally-dense-lat}
For $p \geq 1$ and $n, N \in \Z^+$, a \emph{$(p, n, N)$-exact locally dense lattice} consists of a lattice $\lat \subset \R^m$ of rank $n$ and a target vector $\vec{t} \in \R^m \setminus \lat$ satisfying the following properties.
\begin{enumerate}    \item (Many closest vectors.) $|\{\vec{x} \in \lat : \|\vec{x} - \vec{t}\|_p = \dist_p(\vec{t}, \lat)\}| \geq N$. 
    \item \label{item:distance_less_than_lambda1} (Distance to $\vec{t}$ at most $\lambda_1$.) $\dist_p(\vec{t}, \lat) \le \lambda_1^{(p)}(\lat)$.
    \item \label{local-dense-item-three} (Distance of integer multiples $k\vec{t}$.) For every integer $k$, either $k\vec{t} \in \lat$ or $\dist_p(\vec{t}, \lat)\le \dist_p(k\vec{t},\lat)$. 
\end{enumerate}
\end{definition}
That is, there are at least $N$ closest vectors in $\lat$ to $\vec{t}$, and these vectors are at distance at most $\lambda_1^{(p)}(\lat)$ to $\vec{t}$. Moreover, we require integer multiples $k\vec{t}$ of the target vector $\vec{t}$ to be no closer than $\vec{t}$ to the lattice $\lat$ (excluding the degenerate case when $k \vec{t} \in \lat$, which includes when $k = 0$).

One can easily check that \cref{item:distance_less_than_lambda1,local-dense-item-three} of \cref{def:locally-dense-lat} are precisely equivalent to the statement that $\lambda_1^{(p)}(\lat + \Z\vec{t}) = \dist_p(\vec{t},\lat)$. (Here, we are using the fact that $\lat + \Z \vec{t}$ is in fact a lattice. One can easily show that \cref{local-dense-item-three} implies this, as we do in \cref{lem:well-definedness of lattice}.) So, together the conditions in \cref{def:locally-dense-lat} are precisely equivalent to the statement that there are at least $N$ distinct vectors of the form $\vec{y} - \vec{t} \in \lat + \Z \vec{t}$ such that $\|\vec{y} - \vec{t}\|_p = \lambda_1^{(p)}(\lat + \Z \vec{t})$. (Of course, we have chosen the definition to satisfy this.) From this, we immediately conclude that $\kappa_p(\lat + \Z \vec{t}) \geq N$.\footnote{The mapping $(\lat, \vec{t}) \mapsto \lat + \Z \vec{t}$ for a vector $\vec{t}$ is very closely related to Kannan's embedding.
More precisely, ``Kannan's embedding'' usually refers to the mapping
\[
\basis, \vec{t} \mapsto \begin{pmatrix} \basis & \vec{t} \\ \vec{0} & u \end{pmatrix}  \ \text{,}
\]
where $\basis$ is a lattice basis, $\vec{t}$ is a (target) vector, and $u > 0$ is an orthogonal component. With our requirements for $\vec{t}$, appending the orthogonal component $u$ is unnecessary.} (See \cref{thm:kannans-embedding-formal} for a formal statement of this.)

So, with the right definition in hand, it is immediate that in order to construct lattices with large kissing number, it suffices to construct an exact locally dense lattice. In fact, we show that exact locally dense lattices can themselves be constructed from lattices with high kissing number, so that the two objects are equivalent up to a small loss in parameters. Specifically, we have the following result, which we prove in \cref{app:going_backwards}.

\begin{theorem} \label{thm:LDL-KN-equivalence}
    For any $p$ and $n$, let $N_p(n)$ be the maximum value such that there exists a $(p,n,N_p(n))$-exact locally dense lattice. Then,
    \[
        \kappa_p(n)/2 < N_p(n) \leq \kappa_p(n+1)
        \; .
    \]
\end{theorem}

\paragraph{The usefulness of exact locally dense lattices.}
The reader might wonder whether this connection between exact locally dense lattices and the lattice kissing number is useful.
In fact, it is. Locally dense lattices (though not necessarily \emph{exact} locally dense lattices) have been used extensively in the literature on computational lattice problems. More importantly from our perspective, there are by now many different known constructions of locally dense lattices~\cite{ajtaiGeneratingHardInstances1996,caiApproximatingSVPFactor1999,ajtaiGeneratingHardInstances1999,micciancioShortestVectorLattice2000,khotHardnessApproximatingShortest2005,liuBoundedDistanceDecoding2006a,journals/toc/Micciancio12,conf/coco/Micciancio14,AS18,bennettHardnessBoundedDistance2020a,bhattacharyyaParameterizedIntractabilityEven2021,conf/innovations/BennettPT22,conf/random/BennettP23,bennettParameterizedInapproximabilityMinimum2024,aggarwalMindGapNot2026,hiraharaComplexityLocallyDense2026,wan2026nphardnesssvpeuclideanspace,wanEuclideanSVPDeterministically2026}. 

Intuitively, a lattice-target pair $(\lat, \vec{t})$ is ``locally dense'' if there are many lattice vectors that are remarkably close to $\vec{t}$, and there are many different ways that this is made precise in the literature (often implicitly). The specific term ``locally dense lattice'' was coined by Micciancio~\cite{conf/coco/Micciancio14}, and we briefly explain how our definition of \emph{exact} locally dense lattices differs from Micciancio's definition in \cite{conf/coco/Micciancio14}. (A similar comparison applies to all variants of the definition in the literature.)
 The most significant difference between our definition and that of Micciancio is that we require $\lat$ to contain at least $N$ vectors whose distance to $\vec{t}$ is \emph{exactly} $\dist_p(\vec{t},\lat)$,
whereas~\cite{conf/coco/Micciancio14} just requires there to be $N$ vectors in $\lat$ that are within distance $r$ of $\vec{t}$ for some $\dist_p(\vec{t},\lat) \leq r < \lambda_1^{(p)}(\lat)$. (Equivalently, we require $r = \dist_p(\vec{t},\lat)$, whereas~\cite{conf/coco/Micciancio14} allows for larger $r$.) And, Micciancio did not just require that $r \leq \lambda_1^{(p)}(\lat)$ but instead was interested more generally in the ratio $\gamma := \lambda_1^{(p)}(\lat)/r \geq 1$ and particularly the case when $\gamma \geq 1 + \Omega(1)$ is bounded away from one. These two distinctions are why we use the qualifier ``exact'' in our definition of locally dense lattices.\footnote{Micciancio did not have any requirement analogous to \cref{local-dense-item-three} in \cref{def:locally-dense-lat}, but in applications of locally dense lattices in the literature, $\vec{t}$ is padded with an extra coordinate to effectively enforce \cref{local-dense-item-three}. This padding construction is less natural in our setting in which there is not necessarily any gap between $\lambda_1^{(p)}(\lat)$ and $\dist_p(\vec{t},\lat)$, which is why we simply included \cref{local-dense-item-three} in our definition.
The definition in~\cite{conf/coco/Micciancio14} also includes a matrix $\vec{T}$, which is irrelevant in our context, and the existence of many close vectors in his definition is in some sense witnessed by $\vec{T}$. (The image under $\vec{T}$ of the closest vectors is required to contain a large set.)}

\paragraph{Constructing exact locally dense lattices.}
So, our definition is not identical to the one used in prior work, but it is similar enough that we can easily show that some prior constructions of locally dense lattices in the literature in fact are \emph{exact} locally dense lattices. Specifically, our construction of lattices with $\kappa_p(\lat) \gtrsim e^{\sqrt{n}}$ for all $p$ then follows immediately from noting that the  locally dense lattices constructed by Bennett and Peikert~\cite{conf/random/BennettP23} are in fact already \emph{exact} locally dense lattices. (See \cref{sec:Craig}.) Our explicit deterministic construction of lattices with $\kappa_p(\lat) \gtrsim e^{n^{1/4} \log n}$ for all $p$ follows immediately from observing that Wan's recent breakthrough deterministic construction of locally dense lattices \cite{wan2026nphardnesssvpeuclideanspace} again actually gives exact locally dense lattices. (See \cref{sec:Wan}. In a bit more detail, to prove $\kappa_p(\lat) \gtrsim e^{n^{1/4} \log n}$, we prove a minor endpoint refinement of Wan's result. See \cref{app:Wan_stuff}.)

To obtain lattices $\lat \subset \R^n$ with exponentially large kissing number for $p >2$, we start with the construction of Aggarwal and Stephens-Davidowitz~\cite{AS18}, who essentially showed that a suitably chosen random sublattice of $\Z^n$ together with the target vector $\vec{t} = (1/2,\ldots, 1/2)$ forms a locally dense lattice for any $p > p_0 \approx 2.1397$ (using ideas from \cite{mazoLatticePointsHighdimensional1990,elkiesPackingDensitiesSuperballs1991} to study the geometry of $\Z^n$ and ideas from \cite{khotHardnessApproximatingShortest2005,dadushLatticeSparsificationApproximate2013} to study the geometry of the random sublattice). We show (in \cref{sec:Z^n-local-density}) how to use the same basic idea to obtain exact locally dense lattices for any $p > p_0$. More generally, we show in \cref{thm:sparsification} that any pair $\lat, \vec{t}$ satisfying a certain technical condition can be converted into an exact locally dense lattice. By plugging in $\Z^n$ and $\vec{t} = (1/2,\ldots, 1/2)^\top$, we show lattices with exponential kissing number exist for all $p > p_0$.\footnote{Admittedly, we are oversimplifying here. \cite{AS18} did not \emph{quite} build a locally dense lattice from a random sublattice of $\Z^n$. And, in our setting, we must carefully shift the target $\vec{t}$ to ensure that $2\vec{t}$ lies in the random sublattice. This shift causes some additional complications in the technical details of our proof. Our specific analysis of the closest vectors in this random sublattice to this shifted target vector follows \cite{lattices_parsification,conf/innovations/BennettPT22}.}

However, this construction provably fails for $p \leq p_0$. (See \cite{AS18}.) To obtain lattices with exponential kissing number for all $p > 2$, we therefore generalize this construction by replacing $\Z^n$ with the direct sum $\lat^\ell := \lat \oplus \cdots \oplus \lat \subset \R^{\ell m}$ of some lattice $\lat \subset \R^m$ in fixed dimension $m$. (Notice that $\Z^n$ is the special case when $m = 1$.) And, we take as our target $\vec{t}^\ell := (\vec{t},\ldots, \vec{t}) \in \R^{\ell m}$. The specific lattice that we choose is a Hadamard rotation of the lattice $D_m$ with target vector the same Hadamard rotation of $\vec{e}_1$. (This is a relatively natural choice in this context. In fact, for $m = 1$ and $m=2$, this yields an isomorphic construction to the $\Z^n$-based construction.) We then show how to generalize to this setting the techniques of \cite{AS18} (which are based on \cite{mazoLatticePointsHighdimensional1990,elkiesPackingDensitiesSuperballs1991,khotHardnessApproximatingShortest2005,dadushLatticeSparsificationApproximate2013}). Indeed, these techniques show that a random sublattice of $\lat^\ell$ yields an exact locally dense lattice for any $p > 2$ with $N \geq \Omega(\alpha_p^n)$ for some $\alpha_p > 1$ (see~\cref{sec:finishing} for the definition of $\alpha_p$). As an immediate corollary, we obtain exponentially large kissing number for all $p > 2$.

We also note that this gives the only construction of a locally dense lattice of any kind with exponentially large $N$ for $2 < p \leq p_0 \approx 2.1397$ and a slightly better exponential for $p_0 < p < p_1 \approx 2.1433$. ($p_1$ is the point where the $\Z^n$-based construction achieves the same value as the best non-trivial $D_m$-based construction. See \cref{fig:high_p_plot}.) 

One cannot hope to use the same techniques to prove that lattices with exponential kissing number exist for $p = 2$ or more generally for $1 \leq p \leq 2$. At a technical level, this is because the Theta function $\Theta_p$ that we work with in \cref{sec:sparsification} is positive definite for $p \leq 2$. This implies that for any fixed lattice $\lat \subset \R^m$ and any sufficiently large $\ell$, the number of vectors in a ball around $\vec0$ is never much smaller than the number of vectors in a ball of the same radius around $\vec{t}$. See \cite{mazoLatticePointsHighdimensional1990} for a proof when $\lat = \Z$ and $p = 2$, which generalizes to arbitrary lattices and $p \leq 2$. See also \cite{elkiesPackingDensitiesSuperballs1991,AS18}.

\begin{remark}
We note that our new construction of exact locally dense lattices implies new hardness results for $\SVP_p$, using the techniques in~\cite{AS18}. (We do not attempt to prove these hardness results formally, as this is not the focus of the present work.) Specifically our construction of exact locally dense lattices with exponential $N$ for $p > 2$ (in \cref{cor: alpha to kissing number}, which is the key step in proving \cref{item:main_theorem_exp} of \cref{thm:kn-lower-bounds-main}) implies (randomized) SETH-hardness of $\SVP_p$ for all $p > 2$, rather than just for $p > p_0 \approx 2.1397$. It also implies slightly improved SETH-hardness for $p_0 < p < p_1 \approx 2.1433$, since our construction achieves larger $N$ than the construction in \cite{AS18} for such $p$. (Here $p_0$ and $p_1$ are the values of $p$ described above; see \cref{fig:high_p_plot}.)

Additionally, we note that the (randomized) construction in \cref{thm:craig} of $(p, n, N)$-\emph{exact} locally dense lattices for any $p$ (including $p = 2$) and $N \geq 2^{\Omega(\sqrt{n})}$ implies $2^{\Omega(\sqrt{n})}$-hardness of $\SVP_p$ under \emph{non-gap} randomized ETH. This construction was essentially already present in \cite{conf/random/BennettP23} and even earlier work may yield constructions of (potentially non-exact) locally dense lattices with similar parameters (which would lead to hardness under a gap version of ETH). To the best of our knowledge, no prior work has shown such hardness for $\SVP_2$.
\end{remark}

\subsection{Concurrent Work}
\label{sec:concurrent}

Independent and concurrent work of Laarhoven and Kominers~\cite{laarhoven2026latticefamilykissingconstant} proved a very similar result to \cref{item:main_theorem_subexp} of \cref{thm:kn-lower-bounds-main}. Specifically, they proved that $\kappa_2(n) \geq e^{2\sqrt{n} - o(\sqrt{n})}$. Formally, this result is incomparable to ours, as we obtain the slightly weaker lower bound of $e^{\sqrt{n} - o(\sqrt{n})}$ but include the (simple) observation that this lower bound applies for $\kappa_p(n)$ for all $p \geq 1$, rather than just $p = 2$. 

The two constructions are very similar. In fact, an earlier version of \cite{laarhoven2026latticefamilykissingconstant} proved that $\kappa_2(n) \geq e^{\sqrt{n} - o(\sqrt{n})}$ using an essentially identical construction to our own. The extra factor of two in the exponent is obtained by an elegant modification of the Craig lattices that we use (and that the earlier version of \cite{laarhoven2026latticefamilykissingconstant} used), which has no analogue in our paper.
The authors similarly found this result by prompting AI tools. (See \cref{sec:AI}.)

Our presentation of this result in terms of exact locally dense lattices is novel, and it is this perspective that led to our exponential lower bound on $\kappa_p(n)$ for all $p > 2$.

\subsection{AI Statement}
\label{sec:AI}

The idea behind the use of Craig lattices in the construction of lattices $\lat \subset \R^n$ with $\kappa_2(\lat) \geq 2^{\Omega(\sqrt n)}$ in \cref{thm:craig} came from the second-named and third-named authors prompting ChatGPT.
Subsequently, the authors together obtained most of the remaining results without the use of AI. Specifically, the following results were obtained without the use of AI:
\begin{enumerate}
    \item The abstraction of this construction in terms of exact locally dense lattices (\cref{thm:kannans-embedding-formal}) and the generalization of this theorem given in \cref{thm:sparsification}.
    \item The explicit construction of lattices $\lat \subset \R^n$ with $\kappa_2(\lat) \geq 2^{\Omega(n^{1/4} \cdot \log n)}$ using this abstraction and the work of Wan~\cite{wan2026nphardnesssvpeuclideanspace} in \cref{cor:explicit-construction}.
    \item The invocation of \cref{thm:sparsification} using the integer lattice $\Z^n$ and point-counting techniques from~\cite{mazoLatticePointsHighdimensional1990,elkiesPackingDensitiesSuperballs1991,AS18} to obtain lattices $\lat \subset \R^n$ with $\kappa_p(\lat) \geq 2^{\Omega(n)}$ for $p > p_0 \approx 2.1397$.
\end{enumerate}
The idea to use a rotation of $D_m$ as the base lattice to achieve $\kappa_p(n) \geq 2^{\Omega(n)}$ for $2 < p \leq p_0$ in~\cref{thm: general G and A} also came from the authors.
One of the authors prompted ChatGPT for suggestions about which rotation to use, and it suggested several candidates including the Hadamard matrix $H_m$, which we eventually used.

The paper itself was written with limited use of AI tools. ChatGPT and Claude were used only to help find typos and for help with the Mathematica code used to generate \cref{fig:high_p_plot}.

\section{Preliminaries}
We denote vectors in $\R^m$ by bold lower-case letters (e.g., $\vec{x}$) and matrices in $\R^{m\times n}$ by bold upper-case letters (e.g., $\basis$). 

For a real number $1\le p<\infty$, we define the $\ell_p$ norm of a vector $\vec{x}\in\R^m$ as
\[
\|\vec{x}\|_p := \biggl(\sum_{i=1}^m|x_i|^p\biggr)^{1/p}.
\]
The $\ell_\infty$ norm of $\vec{x}$ is $\|\vec{x}\|_\infty := \max_{1\le i\le m}|x_i|$.
We will use the following standard inequality about norms.
\begin{fact}\label{fact: norm conversion}
For $2\le p<\infty$ and $\vec{x}\in\R^m$, we have
\[
\|\vec{x}\|_p\ge m^{1/p-1/2}\|\vec{x}\|_2.
\]
\end{fact}

\subsection{Lattices}
A lattice $\lat\subseteq\R^m$ of rank $n$ is represented by a matrix $\basis\in \R^{m\times n}$ with full column rank. Specifically,
\[
\lat :=\lat(\basis)= \basis\Z^n.
\]
We define the minimum distance of a lattice $\lat$ in $\ell_p$ norm as 
\[
\lambda_1^{(p)}(\lat) = \min_{\vec{v}\in\lat\setminus\{\bm0\}}\|\vec{v}\|_p.
\]
The $\ell_p$-distance of a target vector $\vec{t}\in\R^m$ to a lattice $\lat\subseteq\R^m$ is
\[
\dist_p(\vec{t},\lat)=\min_{\vec{v}\in\lat}\|\vec{v}-\vec{t}\|_p.
\]
We write
\[
\CVP_p(\vec{t},\lat)=\argmin_{\vec{v}\in\lat}\|\vec{v}-\vec{t}\|_p
\]
for the set of closest vectors in $\lat$ to $\vec{t}$ in $\ell_p$ distance.

The next lemma tells us the lattice in~\cref{thm:kannans-embedding-formal} is well-defined.
\begin{lemma}\label{lem:well-definedness of lattice}
    If $(\lat,\vec{t})$ is a $(p,n,N)$-exact locally dense lattice then $\lat+\Z\vec t$ is a lattice.
\end{lemma}
\begin{proof}
     If $\vec{t}\in\lat$, the proof is immediate. When $\vec{t}\notin\lat$, \cref{def:locally-dense-lat} implies $0<\dist_p(\vec{t},\lat)\le \dist_p(k\vec{t},\lat)$ for all $k\ge 2$ with $k\vec{t}\notin\lat$. For all $\vec{v}\in\lat,z\in\Z$, without loss of generality we assume $z\ge 0$. If $z\vec{t}\in\lat$ then $\|\vec{v}-z\vec{t}\|_p\ge\lambda_1^{(p)}(\lat)\ge \dist_p(\vec{t},\lat)$, unless $\vec v-z\vec t=\bm{0}$. Otherwise,
     \[
     \|\vec{v}-z\vec{t}\|_p\ge \dist_p(z\vec{t},\lat)\ge \dist_p(\vec{t},\lat).
     \]
     This proves $\lat+\Z\vec{t}$ is discrete and additive, and hence it is a lattice.
\end{proof}

The $\ell_p$-kissing number, $\kappa_p$, of a lattice tells us how many lattice vectors simultaneously achieve this minimum in $\ell_p$ norm:
\[
\kappa_p(\lat) = |\{\vec{v}\in\lat:\|\vec{v}\|_p = \lambda_1^{(p)}(\lat)\}|.
\]
We also define $\kappa_p(n) := \max_{\lat \, : \, \rank(\lat) = n} \kappa_p(\lat)$.

\subsection{A bound on binomial coefficients} 
We record a standard estimate of binomial coefficients and a particular choice of parameters that we will use in the sequel.
\begin{fact}\label{fact:binom}
    For $n\ge k\ge 1$ and $k=o(n)$,
    \[
    \binom{n}{k}=\Theta\left(\frac1{\sqrt k}\left(\frac{ne}{k}\right)^ke^{-\frac{k^2}{2n}(1+o(1))}\right)\;.
    \]
    In particular, if $k=\sqrt n\pm O(1)$, then 
    \[
    \binom{n}{k}=\Theta(n^{k/2-1/4}e^k)
    \; .
    \]
\end{fact}

\subsection{Hadamard matrices}
We use the following matrix, which is called a \emph{Hadamard matrix}. 
\begin{definition}[The Hadamard matrix]\label{def:Hadamard}
    Let $H_1=[1]$. For $m\ge 2$ a power of 2, the Hadamard matrix $H_m$ is recursively defined as
   \[H_{m}=
\begin{pmatrix}
H_{m/2} & H_{m/2} \\
H_{m/2} & -H_{m/2}
\end{pmatrix}.
    \]
\end{definition}

We will use the following fact about Hadamard matrices.

\begin{fact}\label{fact: Hadamard}
    $H_m\in\{\pm 1\}^{m\times m}$ satisfies $H_m^\top H_m=m\bm{I}_m$, and 
    \[
    \|H_m\vec e_i\|_p=m^{1/p}
    \]for any $p \geq 1$ and $1\le i\le m$, where $\vec e_i$ is the $i$th standard basis vector.
\end{fact}

Additionally, we will use the following lemma.

\begin{lemma}\label{lem:Hamadard norm}
    For any $p\ge 2$, $\vec x\in \R^m$ where $m$ is a power of 2, we have
    \[
    \|H_m\vec x\|_p\ge m^{1/p}\|\vec x\|_2\ .
    \]
\end{lemma}
\begin{proof}
We have $\|H_m\vec x\|_p\ge m^{1/p-1/2}\|H_m\vec x\|_2=m^{1/p-1/2}\sqrt m\|\vec x\|_2=m^{1/p}\|\vec x\|_2$, where the first inequality follows from~\cref{fact: norm conversion}, and the first equality follows from~\cref{fact: Hadamard}. 
\end{proof}
\subsection{Craig lattices and Reed-Solomon lattices}

One important family of lattices that we will study is the family of \emph{Craig lattices}~\cite{Craig_1978} (see also~\cite[Ch. 8, Sec. 6]{CS99}).
\begin{definition}\label{def:craig}
    For an odd prime $q$ and a positive integer $m$, the \emph{Craig lattice} $A_{q-1}^{(m)}$ is
    \[
A_{q-1}^{(m)}
:=
\set*{
\vec{x}\in\Z^q :
\sum_{i=1}^q x_i = 0 \quad\text{and}\quad 
\sum_{i=1}^q i^j \cdot x_i \equiv 0 \Mod{q} \text{ for $j=1,\ldots,m-1$}} \ \text{.}
\]
\end{definition}
Equivalently, $A_{q-1}^{(m)}$ is (the coefficient embedding of) the principal ideal $((x-1)^m)$ of the quotient ring $\Z[x]/(x^q - 1)$.
The next theorem lower bounds the minimum distance of the Craig lattice. The result follows from~\cite[Chapter 8, Theorem 7]{CS99} and~\cite[Theorem 3.1]{conf/random/BennettP23}.
\begin{theorem}\label{thm:minimum bound}
    If $p\ge 1$ and $m\le q/2$, then $\lambda_1^{(p)}(A_{q-1}^{(m)})\ge (2m)^{1/p}$.
\end{theorem}

Reed-Solomon lattices are another family of lattices whose definition is very similar to that of Craig lattices.
\begin{definition}[Reed-Solomon lattices]
\label{def:RS}
For an odd prime $q$ and an integer $1 \leq m \leq q$, the \emph{Reed-Solomon lattice} $\lat_{q,m}$ is
\[
\lat_{q,m} :=
\set*{
\vec{x}\in\Z^q :
\sum_{i=1}^q i^j \cdot x_i \equiv 0 \Mod{q} \text{ for $j=0,\ldots,m-1$}} \ \text{.}
\]
\end{definition}

Equivalently, one can obtain a Reed-Solomon lattice $\lat_{q,m}$ by applying Construction A to a Reed-Solomon code over $\F_q$ with block length $q$ and codimension $m$.

Note that the only difference between the definition of a Craig lattice and a Reed-Solomon lattice is that the Craig lattice definition requires $\sum_{i=1}^q x_i=0$ while the latter only imposes such a condition modulo $q$. 
In particular, the Craig lattice $A_{q-1}^{(m)}$ is the codimension-one sublattice of $\lat_{q,m}$ obtained by intersecting $\lat_{q,m}$ with the subspace orthogonal to $(1,\ldots, 1)^\top$.

\subsection{Counting short lattice vectors using the Theta function}

For a discrete set $S\subset \R^n$ and radius $r>0$, we write 
\[
N_p(S,r) := |\{\vec{y}\in S:\|\vec{y}\|_p\le r\}|,\quad
N_p^o(S,r) := |\{\vec{y}\in S:\|\vec{y}\|_p< r\}|.
\]
We define the Theta function as
\[
\Theta_p(\lat,\tau):=\sum_{\vec{z}\in\lat}e^{-\tau \|\vec{z}\|_p^p}\ ,
\]
where $\lat$ is a lattice in $\R^n$, $\tau>0$ and $p\ge 1$. We have the following simple result (following~\cite[Theorem~4.3]{AS18}, which itself follows \cite{mazoLatticePointsHighdimensional1990,elkiesPackingDensitiesSuperballs1991}).
\begin{theorem}\label{thm:ball_point_count}
For any lattice $\lat\subseteq \R^n$, any radius $r>0$, and any $\tau > 0$,
\[
    N_p(\lat,r) < \min_{\tau >0} e^{\tau r^p}\Theta_p(\lat,\tau)
    \; .
\]
In particular, $N_p(\Z^n,n^{1/p}/2) < c_p^n$ where
\[
c_p:=\min_{\tau>0}e^{\tau/2^{p}}\sum_{z\in\Z}e^{-\tau|z|^{p}}
\; .
\]
Furthermore, when $p > p_0 $, then $0<c_p < 2$, where
 $p_0 \approx 2.1397$ is defined as the unique value such that $c_{p_0} = 2$.
\end{theorem}
\begin{proof}
 We use the following observation in~\cite{AS18}. For $r>0$, $\Theta_p(\lat,\tau)> \sum_{\vec{y}\in \lat,\|\vec{y}\|_p\le r}e^{-\tau\|\vec{y}\|_p^p}\ge N_p(\lat,r) e^{-\tau r^p}$, and hence $N_p(\lat,r)<e^{\tau r^p}\Theta_p(\lat,\tau)$. Taking the minimum over $\tau>0$ (it is easy to see that the minimum is achieved), we get
\[
    N_p(\lat,r)< \min_{\tau>0} e^{\tau r^p}\Theta_p(\lat,\tau)\ .
\]
For $\Z^n$, we also have that
\begin{align}
\Theta_p(\Z^n,\tau)=\sum_{\vec{z}\in\Z^n}e^{-\tau\|\vec{z}\|_p^p}&=\sum_{z_1,\dots,z_n\in\Z}e^{-\tau\sum_{i=1}^n|z_i|^p}\notag\\
&= \left(\sum_{z\in \Z}e^{-\tau |z|^p}\right)^n\notag\\
&=\Theta_p(\Z,\tau)^n. \label{eq:orthogonal-to-power-theta}
\end{align}
This implies $N_p(\Z^n,r)< \left(\min_{\tau>0} e^{\tau r^p/n}\Theta_p(\Z,\tau) \right)^n$.

Fix $r=n^{1/p}/2$. Let $g(p,\tau)=e^{\tau/2^{p}}\sum_{z\in\Z}e^{-\tau|z|^{p}}$, where $p\ge 1,\tau>0$. Let $c_p:=\min_{\tau>0}g(p,\tau)$, we have
\[
N_p(\Z^n,\frac12 n^{1/p})< \left(\min_{\tau>0} e^{\tau /2^p}\Theta_p(\Z,\tau) \right)^n=\left(\min_{\tau>0}g(p,\tau)\right)^n=c_p^n.
\]
Since $|z|=0$ or $|z|\ge 1$, for a fixed $\tau$, $g(p,\tau)$ is strictly decreasing with $p$. $\min_{\tau>0}g(p,\tau)=2$ is achieved at $p_0\approx 2.1397$ (which can be easily computed to high accuracy, as is discussed in \cite{AS18}), and monotonicity shows that $p_0$ is unique. This proves that for any $p>p_0$, $c_p=\min_{\tau>0}g(p,\tau)<\min_{\tau>0}g(p_0,\tau)=2$.
\end{proof}

\subsection{Lattice sparsification}

We next present a result about lattice sparsification, which shows that if a lattice $\lat$ and target $\vec{t}$ in some sense ``almost'' form a locally dense lattice, then a suitably chosen random sublattice $\lat'$ and shifted target $\vec{t}'$ will form an exact locally dense lattice with high probability. This result is adapted from~\cite[Proposition~2.5]{conf/innovations/BennettPT22}. It differs from \cite[Proposition~2.5]{conf/innovations/BennettPT22} primarily in (1) the need to consider the order $s$ of $\vec{t}$ modulo $\lat$ and (2) the related need to choose $\vec{t}'$ carefully. This causes some additional complications in the proof. As the proof is a bit long and relatively similar to prior work, we defer it to~\cref{app:A}.

\begin{restatable}[Sparsification with a careful shift]{proposition}{prop}\label{lem:sparsification-25}
    Let $p\ge 1$, $n\ge 2$, $\lat$ be a lattice of rank $n$ with basis $\basis$, $\vec{t}\in\spn(\lat)\setminus \lat$. For $r>0$, let $$A=|\{\vec{y}-k\vec{t}:\vec{y}\in\lat,k\in\Z,\|\vec{y}-k\vec{t}\|_p<r\}|<\infty,\quad G=|\{\vec v\in \lat:\|\vec v-\vec t\|_p=r\}|.$$ Let $q>A$ be an odd prime, $s=\min\{z\in\Z_{\ge 1}:z\vec{t}\in\lat\}$, and $\vec d\in\Z^n$ be such that $\basis\vec d = s\vec t$.

    Assume $q\nmid s$, sample $\vec{a}$ uniformly over $\Z^n_q$ and fix $\vec z = s^*\vec d$ where $s^*$ is any representative of $s^{-1}\pmod q$ in $\Z$. Define 
    \[
    \lat' = \{\basis \vec{x}:\vec{x}\in\Z^n,\langle \vec{x},\vec{a}\rangle\equiv 0\pmod q\},\quad \vec{t}':=\vec{t}-\basis\vec{z}.
    \]
    Then $\lat'$ has rank $n$, $\lat'+\Z\vec t'$ is a lattice, and the following hold:
    \begin{enumerate}
        \item \label{item-one-BPT25}(Minimum distance.) $
\Pr[\lambda_1^{(p)}(\lat'+\Z\vec t')<r]\le \frac{A-1}{q}$.
        \item \label{item-two-BPT25} (Close vector count.) For any $\delta>0$, $\Pr[N_p(\lat'-\vec t',r)\le (1-\delta)\cdot\frac{G}{q}]\le\frac{Cq}{\delta^2 G}$, where $C=1$ if $s> 2$, and $C=2$ if $s=2$.
    \end{enumerate}
\end{restatable}
\section{Our main generic constructions}

In this section, we will prove the two main generic results that are behind our constructions of lattices with large kissing number. First, we prove that an exact locally dense lattice does in fact yield a lattice with large kissing number. (As we mentioned in \cref{sec:our_constribution} and as is evident from the short proof, this result is immediate once one has the right definition of an exact locally dense lattice.)

\begin{theorem}[Kissing number from exact locally dense lattice]
\label{thm:kannans-embedding-formal}
Fix any $p\ge 1$, let $(\lat, \vec{t})$ be a $(p, n, N)$-exact locally dense lattice. Then $\lat+\Z\vec t$ is a lattice with $\kappa_p(\lat+\Z\vec t) \ge N$.
\end{theorem}
\begin{proof}
\cref{lem:well-definedness of lattice} shows that $\lat+\Z\vec t$ is a lattice; in particular, it is discrete. For every $\vec w =\vec v+z\vec t$ for $\vec v\in\lat,z\in \Z$, if $z\vec t\in\lat$ then $\vec w\in\lat$ and $\|\vec w\|_p\ge\lambda_1^{(p)}(\lat)\ge\dist_p(\vec t,\lat)$ (unless $\vec w=\bm{0}$). Otherwise, $\|\vec w\|_p\ge\dist_p(-z\vec t,\lat)=\dist_p(|z|\vec t,\lat)\ge\dist_p(\vec t,\lat)$. Since $\dist_p(\vec t,\lat)$ is achieved by vectors in $\CVP_p(\vec t,\lat)-\vec t$, which has size at least $N$, this proves $\lambda_1^{(p)}(\lat+\Z\vec t) = \dist_p(\vec{t},\lat)$ and $\kappa_p(\lat+\Z\vec t) \geq N$.
\end{proof}

Next, we show a way to construct exact locally dense lattices from seemingly much less rigid objects. In particular, suppose that the number of vectors in a lattice $\lat$ at a certain distance $r$ to $\vec{t}$ is much larger than the number of vectors of length less than $r$ in $\lat + k\vec{t}$ for all $k$. (Notice that exact locally dense lattices are precisely the special case for which there are $N$ vectors at distance $r$ and no non-zero vectors in $\lat + k \vec{t}$ of length less than $r$.) Then, we can apply the sparsification procedure of~\cite{khotHardnessApproximatingShortest2005,dadushLatticeSparsificationApproximate2013,lattices_parsification,conf/innovations/BennettPT22}. With decent probability, this removes the short lattice vectors, while preserving a large number of close vectors (to a modified target), and hence produces an exact locally dense lattice.

\begin{theorem}[Exact locally dense lattice from sparsification]\label{thm:sparsification} 
Fix any $p\ge 1$. For all sufficiently large $n$, let $\lat\subseteq \R^m$ be a rank $n$ lattice, and $\vec{t}\in\mathrm{span}(\lat)\setminus\lat$ be a target vector.  For $r>0$, define $A = A(\lat,\vec t,r)$ and $G = G(\lat,\vec t,r)$ as
$$A :=|\{ \vec{y}-k\vec{t} :\vec{y}\in\lat, k \in \Z, \|\vec{y} - k \vec{t}\|_p < r\}|<\infty,\quad G :=|\{\vec{x}\in\lat:\|\vec x-\vec t\|_p=r\}| \ \text{.}$$
If $G/A\ge 64$, then there exists a rank $n$ sublattice $\lat'\subseteq\lat$ and target $\vec{t}'$ that form a $(p,n,G/(8A))$-exact locally dense lattice. In particular, $\kappa_p(\lat'+\Z\vec t') \geq G/(8A)$ by \cref{thm:kannans-embedding-formal}.
\end{theorem}
\begin{proof}
Let $q$ be a prime in the range $(2A,4A)$.
Such a prime is guaranteed to exist by Bertrand's postulate.

We apply the sparsification result in~\cref{lem:sparsification-25} to obtain the sublattice $\lat'\subseteq\lat$ and target $\vec t'$. In particular, we have that $\lat'+\Z\vec t'$ is a lattice. More interestingly,
\[
\Pr[\lambda_1^{(p)}(\lat'+\Z\vec t')<r]\le \frac{A-1}{q}
\]
and
\[\Pr\Big[N_p(\lat'-\vec t',r)\le \frac{G}{2q} \Big]<\frac{8q}{G}
\; ,
\]
where we take $\delta=\frac12$ in~\cref{item-two-BPT25} of \cref{lem:sparsification-25}. This then implies that with probability at least $1-\frac{A-1}{q}-\frac{8q}{G}>\frac12-\frac{32A}{G}\ge 0$, all of the following hold:
\begin{enumerate}
    \item $\dist_p(\vec t',\lat')=r\le \lambda_1^{(p)}(\lat')$.
    \item $|\CVP_p(\vec t',\lat')|\ge \frac{G}{2q}>\frac{G}{8A}$.
    \item For integer $k$, either $k\vec t'\in\lat'$ or $\dist_p(\vec t',\lat')=r\le \dist_p(k\vec t',\lat')$.
\end{enumerate}
In other words, with non-zero probability, $\lat'$ and $\vec{t}'$ form a $(p,n,G/(8A))$-exact locally dense lattice.
\end{proof}
\section{Constructions from Craig lattices}

\subsection{Lattices with \texorpdfstring{$\kappa_p(n) \geq 2^{\Omega(\sqrt n)}$ for all $p$}{subexponentially large kissing number for all p}}

\label{sec:Craig}
We study the local density of Craig lattices with a randomly chosen boolean vector $\vec{t} \in \{0,1\}^q$ as the target. Note that for such vectors, the number of nonzero entries is equal to its $\ell_1$ norm.

To that end, suppose that $\vec{s},\vec{t}\in\{0,1\}^q$ satisfy
\begin{equation}
\|\vec{s}\|_1=\|\vec{t}\|_1,\quad
\sum_{i=1}^q s_i\cdot i^j  \equiv \sum_{i=1}^q t_i\cdot i^j\pmod q\quad j=1,\dots,m-1 \; . \label{eq:craig}
\end{equation}
Then, by the definition of the Craig lattice $A_{q-1}^{(m)}$ (\cref{def:craig}), notice that $\vec{t} - \vec s\in A_{q-1}^{(m)}$. 
Indeed, the vectors $\vec{s} \in \{0,1\}^q$ satisfying \cref{eq:craig} are exactly the boolean vectors in the coset $A_{q-1}^{(m)} + \vec{t}$.

Furthermore, $\|(\vec t-\vec s)-\vec t\|_p=\|\vec{s}\|_p = \|\vec{t}\|_p$. So, suppose $\vec{t} \in \{0,1\}^q$ is such that $\dist_p(\vec{t}, A_{q-1}^{(m)}) = \|\vec{t}\|_p$, and that there exist a large number of distinct boolean vectors $\vec{s}_1,\dots,\vec{s}_N$ satisfying \cref{eq:craig} with $\vec{t}$. Then, we obtain many closest vectors $\vec{t} - \vec{s}_i \in A_{q-1}^{(m)}$ to $\vec t$. 

Indeed, a simple pigeonhole argument shows that such a $\vec{t}$ must exist with $N \gtrsim e^{\sqrt{q}}$ when $m \approx \sqrt{q}$ is chosen appropriately. The following theorem makes this precise.

\begin{theorem}\label{thm:craig}
For any prime $q \geq 5$, there exists a target $\vec{t}\in\{0,1\}^q$ such that $(A_{q-1}^{(m)},\vec{t})$ is $(p,q-1,\Omega(e^{\sqrt q}q^{3/4}))$-exact locally dense for all $p \geq 1$ with $m :=  \lfloor\sqrt{q}/2\rfloor$.
\end{theorem}
\begin{proof}
Consider the lattice $A_{q-1}^{(m)}$, note $1\le m=\lfloor\frac{\sqrt q}{2}\rfloor\le\frac{q}2$. For any vector $\vec{s}\in\{0,1\}^q$ define $f:\{0,1\}^q\to \Z_q^{m-1}$ as %

\[
f(\vec{s}) = \left(\sum_{i=1}^q s_i\cdot i^j \bmod q\right)_{1\le j<m}.
\] 

Recalling the definition of the Craig lattice $A_{q-1}^{(m)}$ from \cref{def:craig}, we see that for any $\vec{s}, \vec{t} \in \{0,1\}^q$, we have that $\vec{t} - \vec{s} \in A_{q-1}^{(m)}$ if (and only if) $\|\vec{s}\|_1 = \|\vec{t}\|_1$ and $f(\vec{s}) = f(\vec{t})$.
Since all coordinates of $f(\vec{s})$ are in $\Z_q$, we see that there are at most $q^{m-1}$ values for $f(\vec{s})$. It follows from the pigeonhole principle that there must be some fixed value of $f(\vec{s})$ that is achieved by at least
\[\frac{\binom{q}{2m}}{q^{m-1}}
\]
distinct choices of $\vec{s} \in \{0,1\}^q$ with $\|\vec{s}\|_1 = 2m$. Using~\cref{fact:binom} and that $\sqrt q-2<2m\le \sqrt q$, we have 
\[
\binom{q}{2m}\geq \Omega\left(q^{m-\frac14}e^{2m}\right).
\]
Hence 
\[
\frac{\binom{q}{2m}}{q^{m-1}}\geq \Omega\left(\frac{q^{m-\frac14}e^{2m}}{q^{m-1}}\right)\geq\Omega(e^{\sqrt q}q^{3/4}).
\]
Let $\vec{t} \in \bit^q$ be such that $f(\vec{s}) = f(\vec{t})$ and $\|\vec{s}\|_1 = \|\vec{t}\|_1$ for at least $\binom{q}{2m}/q^{m-1}$ elements $\vec{s} \in \bit^q$.
Then for such $\vec{t}$ and $\vec{s}$, we have $\|(\vec{t}- \vec{s}) - \vec{t}\|_p = \|\vec{s}\|_p= (2m)^{1/p}$, so that there are at least $N$ lattice vectors with distance exactly $(2m)^{1/p}$ from $\vec{t}$.

To show that this is an exact locally dense lattice, it therefore suffices to show that $\lambda_1^{(p)}(A_{q-1}^{(m)}) \geq (2m)^{1/p}$ and that $\dist_p( k\vec{t},A_{q-1}^{(m)}) \geq (2m)^{1/p}$ for all $k \geq 1$. In fact, the lower bound on $\lambda_1^{(p)}(A_{q-1}^{(m)})$ is exactly \cref{thm:minimum bound}. To show that $\dist_p( k\vec{t},A_{q-1}^{(m)}) \geq (2m)^{1/p}$, let $\vec{v} \in A_{q-1}^{(m)}$ and consider $\vec{w} := \vec{v} - k\vec{t} \in \Z^q$. Note that $\sum_{i=1}^q v_i = 0$ by the definition of $A_{q-1}^{(m)}$, so that 
\[
\|\vec{w}\|_p^p\ge\|\vec{w}\|_1\ge\bigg|\sum_{i=1}^qw_i\bigg| = \bigg|\sum_{i=1}^qv_i-k(2m)\bigg|=k(2m)\ge 2m
\; ,
\]
as needed.
\end{proof}

We then get the following corollary.

\begin{corollary}\label{cor:craig kn}
   For any prime $q \geq 5$, there exists $\vec{t}\in\{0,1\}^q$ such that $\kappa_p(A_{q-1}^{(m)}+\Z\vec{t})\ge \Omega(e^{\sqrt q}q^{3/4})=2^{\Omega(\sqrt q)}$ for all $p \geq 1$, where $m:=\lfloor\sqrt{q}/2\rfloor$.
\end{corollary}
\begin{proof}
    The result follows by applying~\cref{thm:kannans-embedding-formal} to the exact locally dense lattice of~\cref{thm:craig}.
\end{proof}
 
\begin{remark}[Extension to all $n$]\label{rmk:padding}
We can remove the restriction of~\cref{cor:craig kn} to prime $q$ by a standard padding argument. Suppose that, for every sufficiently large prime $q$, we have constructed a rank $q$ lattice $\lat$ satisfying $\kappa_p(\lat)\ge \Omega(e^{\sqrt q}q^{3/4})$ (note $A_{q-1}^{(m)}+\Z\vec t$ has rank $q$). For any $n\ge q$, we can pad $\lat$ into a rank $n$ lattice $\lat'$ in a way that preserves relevant shortest vectors (e.g., adding long vectors orthogonal to $\lat$). Hence $\kappa_p(n)\ge \kappa_p(\lat') = \kappa_p(\lat)$.
By the prime number theorem, there always exists a prime $q$ with $n-o(n) < q \leq n$, and applying this construction yields 
\[
\kappa_p(n) \geq \Omega(e^{\sqrt{q}}) > e^{\sqrt{n} - o(\sqrt{n})} \ \text{.}
\]
\end{remark}

\subsection{An explicit construction with \texorpdfstring{$\kappa_p(n)\ge 2^{\Omega(n^{1/4}\log n)}$}{kappa(n) >= 2**Omega(n**1/4 log n)} for all \texorpdfstring{$p$}{p}}
\label{sec:Wan}

In his breakthrough proof of deterministic $\NP$-hardness of $\SVP$ in the $\ell_2$ norm~\cite{wan2026nphardnesssvpeuclideanspace}, Wan proved that an explicitly chosen vector $\vec{y}$ is close to a large number of vectors in the Reed-Solomon lattice $\lat_{q,k}$ (as defined in \cref{def:RS}). %

\begin{theorem}[{\cite[Theorem~6.3]{wan2026nphardnesssvpeuclideanspace}}]\label{thm:locally dense construction}
    Choose $0<\varepsilon<1$ and $\varepsilon_1$ such that $0<\varepsilon_1<\frac{\varepsilon}{2(1+\varepsilon)}< \frac14$. Define positive integers $k,h$ as 
    \[
k = \left\lfloor \frac{q^{\varepsilon_1}}{2}\right\rfloor,\quad h = \lfloor(1+\varepsilon)k\rfloor\le q/2.
    \]
    For $p\ge 1$, the Reed-Solomon lattice $\lat_{q,k}$ in $\R^q$ satisfies the following two properties:
    \begin{enumerate}
        \item $\lambda_1^{(p)}(\lat_{q,k})\ge (2k)^{1/p}$, and
        \item Let $\vec{y} = (1,\dots,1,0,\dots,0)^\top$ be the binary vector in $\Z^q$ with Hamming weight $h$. Define
        \[
        S_{2}(\vec y,h):=\{\vec{x}\in(\lat_{q,k}+\vec{y})\cap \{0,1\}^q:\|\vec{x}\|_1=h\}.
        \]
        Then, for all sufficiently large primes $q$, we have $|S_2(\vec y,h)|\ge q^{\varepsilon_1(1+\varepsilon)k}.$
    \end{enumerate}
\end{theorem}
Setting $\varepsilon_1=\frac14-\delta$ for any $0<\delta<\frac14$ and taking $\varepsilon$ close to one such that $\varepsilon_1<\frac{\varepsilon}{2(1+\varepsilon)}$ is satisfied, we obtain $q^{\Omega(q^{1/4-\delta})}$ close vectors to $-\vec y$ in $\lat_{q,k}$. We now 
give an endpoint refinement of this theorem, the proof of which is deferred to~\cref{app}.

\begin{restatable}[Refinement of~{\cite[Theorem~6.3]{wan2026nphardnesssvpeuclideanspace}}]{theorem}{corRS}\label{cor:refinement of wan}
    Fix $p\ge 1$, and take $k=\lfloor\alpha q^{1/4}/2\rfloor$ for some $0<\alpha<1$. Let $\vec{y} = (1,\dots,1,0,\dots,0)^\top$ be the binary vector in $\Z^q$ with Hamming weight $2k$. For all sufficiently large primes $q$, we have
  $|S_2(\vec y,2k)|\ge 2^{\Omega(q^{1/4}\log q)}$, where $S_2$ is as defined in \cref{thm:locally dense construction}.
\end{restatable}

\begin{theorem}\label{thm: Wan locally dense to Craig}
    Using the same parameters as in~\cref{cor:refinement of wan}, $A_{q-1}^{(k)}$ and $\vec y$ form a $(p,q-1,2^{\Omega(q^{1/4}\log q)})$-exact locally dense lattice.
\end{theorem}
\begin{proof}
    Recall $A_{q-1}^{(k)}=\lat_{q,k}\cap\{\vec x\in\R^q:\sum_{i=1}^q x_i=0\}$. For any vector $\vec{u}\in S_2(\vec y,2k)$, we see that $\vec{u}-\vec{y}\in\lat_{q,k}$. Notice $\sum_{i=1}^q(\vec{u}-\vec{y})_i=0$ as $\vec{u},\vec{y}$ are boolean vectors with equal Hamming weight. This implies $\pm(S_2(\vec y,2k)-\vec{y})\subseteq A_{q-1}^{(k)}$. 

    Let $\vec{v}:=\vec{w}-z\vec{y}$ for some $\vec{w}\in A_{q-1}^{(k)}$ and integer $z\ge 1$. Since $\sum_{i=1}^qw_i =  0 $ by the definition, we see that
    \[
    \sum_{i=1}^qv_i= -z\sum_{i=1}^qy_i= -2zk,
    \]
    since $\vec{y}$ has Hamming weight $2k$. Therefore $\|\vec{v}\|_p^p\ge\left|\sum_{i=1}^qv_i\right|= 2zk$, where we use the fact that $\vec{v}\in\Z^q$. For $z=1$ this shows $\dist_p(\vec{y},A_{q-1}^{(k)})\ge(2k)^{1/p}$, with equality obtained by vectors in $\vec{y}-S_2(\vec y,2k)$. For $z\ge 2$, this proves $\dist_p(z\vec{y},A_{q-1}^{(k)})\ge\dist_p(\vec{y},A_{q-1}^{(k)})$. Since~\cref{thm:minimum bound} shows $\lambda_1^{(p)}(A_{q-1}^{(k)})\ge (2k)^{1/p}$, \cref{cor:refinement of wan} implies $(A_{q-1}^{(k)},\vec{y})$ is a $(p,q-1,2^{\Omega(q^{1/4}\log q)})$-exact locally dense lattice, as needed.
\end{proof}
\begin{corollary}[Explicit construction for $p\ge 1$] \label{cor:explicit-construction}
     Using the same parameters as in~\cref{cor:refinement of wan}, $\kappa_p(A_{q-1}^{(k)}+\Z\vec{y})\ge2^{\Omega(q^{1/4}\log q)}$, and such a lattice is deterministically constructed in $\poly(q)$-time.
\end{corollary}
\begin{proof}
    The result follows by applying~\cref{thm:kannans-embedding-formal} to the exact locally dense lattice of~\cref{thm: Wan locally dense to Craig}. The running time is immediate from the construction.
\end{proof}
\begin{remark}
    Using a similar padding argument as~\cref{rmk:padding}, we get explicitly constructed lattices achieving $\kappa_p (n)\ge 2^{\Omega(n^{1/4}\log n)}$ for all sufficiently large $n$. And it is immediate that such lattices can be constructed deterministically in $\poly(n)$-time.
\end{remark}
\section{Constructions from sparsified direct sums}
\label{sec:sparsification}

In this section, we apply~\cref{thm:sparsification}
to obtain lattices $\lat$ with large kissing number $\kappa_p(\lat)$.
Following the notation of~\cref{thm:sparsification}, for a lattice $\lat \subset \R^n$, target $\vec t$ and $r>0$ we define
$$
A = A(\lat,\vec t,r) :=|\{\vec y-k\vec t:\vec y\in\lat,k\in\Z,\|\vec y-k\vec t\|_p<r\}|,\  G = G(\lat,\vec t,r) :=|\{\vec x\in\lat:\|\vec x-\vec t\|_p=r\}|.
$$
Recall in particular that, given such $\lat$ and $\vec{t}$, \cref{thm:sparsification} allows us to construct a lattice $\lat' \subseteq \R^n$ with kissing number $\kappa_p(\lat') \gtrsim G/A$.
For our particular constructions, computing $G$ will be relatively straightforward. To bound $A$, we use Theta functions as in~\cite{mazoLatticePointsHighdimensional1990,elkiesPackingDensitiesSuperballs1991,AS18}. In fact, our construction generalizes an earlier construction from \cite{AS18}. We first quickly present the \cite{AS18} construction in our language.

\subsection{Lattices with \texorpdfstring{$\kappa_p(n)\ge 2^{\Omega(n)}$}{kappa-p(n) >= 2**(Omega(n))} for \texorpdfstring{$p>p_0\approx 2.1397$}{p > p0 approximately 2.1397}}
\label{sec:Z^n-local-density}

In this subsection, we present a simple construction that achieves exponential kissing number. \cite{AS18} already showed how to use the ideas below to construct a certain kind of locally dense lattice. Here, we show that the same ideas yield an \emph{exact} locally dense lattice. The construction that we present in \cref{sec:Dm_rotation_thing} is a strict generalization of this construction, but this section serves as a useful warm up and motivates the more complicated construction in \cref{sec:Dm_rotation_thing}.

For any $p\ge 1$, we write \[\alpha_{p,1}:=2/c_p
\;, \]
where \[
c_p :=\min_{\tau>0}e^{\tau/2^p}\sum_{z\in\Z}e^{-\tau|z|^p}
\]
as defined in~\cref{thm:ball_point_count}. Note that this notation is consistent with the more general definition of $\alpha_{p,m}$ in~\cref{cor: alpha to kissing number}, which is why we use it here.\footnote{This is because when $m=1$, $\M_1$ of~\cref{cor: alpha to kissing number} is exactly $2\Z$, so $\alpha_{p,1}=2/(\min_{\tau>0}e^{\tau}\Theta_p(2\Z,\tau))$. The denominator is equal to $\min_{\tau>0}e^{\tau}(\sum_{z\in\Z}e^{-\tau 2^p\cdot |z|^p})$. Changing variables from $\tau$ to $\tau/2^p$ gives exactly $c_p$, and hence $\alpha_{p,1}=2/c_p$.} 

The following theorem shows that for the lattice $\lat =\Z^n$ and target $\vec{t} = (1/2,\ldots,1/2)^\top$, $G/A \gtrsim \alpha_{p,1}^n$, where $G$ and $A$ are the number of ``good'' and ``annoying'' vectors respectively, as defined in \cref{thm:sparsification}. This is essentially immediate from our definitions and \cref{thm:ball_point_count} (closely following \cite{AS18}). For completeness, we include in \cref{fig:alpha1} a plot of $\alpha_{p,1}$. Also note that as $p \to \infty$, $\alpha_{p,1} \to 2$ (which is immediate from the definition).

\begin{figure}[!ht]
    \def\arraystretch{3}
    \centering
    \begin{tabular}{@{}l@{}}
    \includegraphics[width=0.9\textwidth]{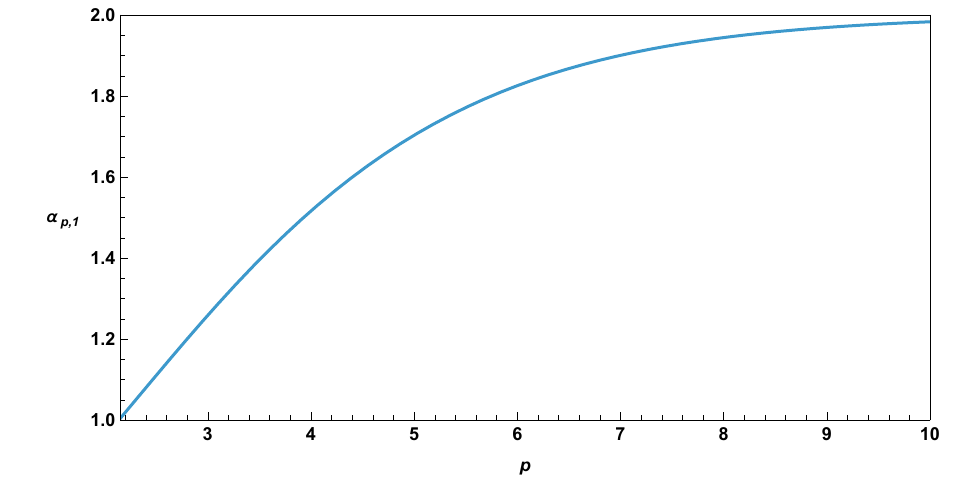}
    \end{tabular}
    \caption[Plot of $\alpha_{p,1}$]{Numerical estimates of $\alpha_{p,1}$ for different values of $p$. This was computed by simply restricting the sum defining $c_p$ to the finite set $z \in [-10,10]$ (which is very accurate). In particular, \cref{thm:large p} shows that $\kappa_p(n) \geq \Omega(\alpha_{p,1}^n)$.}
    \label{fig:alpha1}
\end{figure}

\begin{theorem}\label{thm: new G/A for Zn}
For any $p\ge 1$, and $\vec t := (\frac12,\frac12,\dots,\frac12)^\top\in\R^n$, we have
    \[
    \frac{G(\Z^n,\vec t,n^{1/p}/2)}{A(\Z^n,\vec t,n^{1/p}/2)}\ge \alpha_{p,1}^n \ \text{.} %
    \]
\end{theorem}
\begin{proof}
We have $\CVP_p(\vec{t},\Z^n)=\{0,1\}^n$, with $\dist_p(\vec{t},\Z^n)=n^{1/p}/2$, so $G(\Z^n,\vec t,n^{1/p}/2)=2^n$. \cref{thm:ball_point_count} implies $N^o_p(\Z^n,n^{1/p}/2)< c_p^n$. 
Note that $\Z^n + k\vec{t} = \Z^n + \vec{t}$ for odd integers $k$ and $\Z^n + k\vec{t} = \Z^n$ for even integers $k$.
So,
$$A(\Z^n,\vec t,n^{1/p}/2)=N^o_p(\Z^n,n^{1/p}/2)<c_p^n.$$
Since $\alpha_{p,1}=2/c_p$, this proves $G/A\ge (2/c_p)^n=\alpha_{p,1}^n\ .$
\end{proof}

Applying \cref{thm:sparsification} immediately yields the following result.

\begin{corollary}[Exponential kissing number for $p>p_0$]\label{thm:large p}
  Using the same $\vec t,\alpha_{p,1}$ as~\cref{thm: new G/A for Zn}, when $p>p_0$ where $p_0 \approx 2.1397$ is defined in~\cref{thm:ball_point_count}, we have $\alpha_{p,1}>1$.
  
  In particular, there exists a sublattice $\lat'\subseteq \Z^n$  and $ \vec{t}'\in\R^n$ such that $(\lat',\vec{t}')$ is $(p,n,\Omega(\alpha_{p,1}^n))$-exact locally dense, and $\kappa_p(\lat'+\Z\vec{t}')\ge\Omega(\alpha_{p,1}^n)= 2^{\Omega(n)}$.  
\end{corollary}
\begin{proof}
\cref{thm:ball_point_count} shows $0<c_p<2$ for our choice of $p$, so $\alpha_{p,1}=2/c_p>1$. We apply~\cref{thm:sparsification} with $\lat=\Z^n,\vec t=(\frac12,\frac12,\dots,\frac12)^\top\in\R^n$, and $r=n^{1/p}/2$. \cref{thm: new G/A for Zn} shows $G/A\ge \alpha_{p,1}^n$, so we get that for some sublattice $\lat'\subseteq\Z^n,\vec{t}'\in\R^n$, $(\lat',\vec{t}')$ is $(p,n,\Omega(\alpha_{p,1}^n))$-exact locally dense, and $\kappa_p(\lat'+\Z\vec{t}')\ge \Omega(\alpha_{p,1}^n)=2^{\Omega(n)}$.
\end{proof}

\subsection{Lattices with \texorpdfstring{$\kappa_p(n)\ge 2^{\Omega(n)}$}{kappa-p(n) >= 2**(Omega(n))} for all \texorpdfstring{$p>2$}{p > 2}}
\label{sec:Dm_rotation_thing}

We now show how to generalize the simple $\Z^n$-based construction that we gave in \cref{sec:Z^n-local-density} to a more general direct sum. To that end, we write $\lat^\ell := \{ (\vec{y}_1^\top,\ldots, \vec{y}_\ell^\top)^\top \ : \ \vec{y}_i \in \lat\}$ and $\vec{t}^\ell := \vec{1}_{\ell} \otimes \vec{t} = (\vec{t}^\top, \ldots, \vec{t}^\top)^{\top}$. Our construction will use such a product lattice $\lat^\ell \subset \R^{m\ell}$ and target $\vec{t}^{\ell} \in \R^{m\ell}$ for some carefully chosen $m$-dimensional lattice $\lat \subset \R^m$ and target $\vec{t} \in \R^m$. We will again choose $\lat$ and $\vec{t}$ so that $2\vec{t} \in \lat$, which in particular implies that $2\vec t^\ell\in\lat^\ell$, hence the only ``annoying'' vectors (i.e., the vectors counted by $A$ in \cref{thm:sparsification}) are simply the short vectors in $\lat^\ell$. (In particular, if $k$ is odd, then $\lat^\ell - k\vec{t}^\ell = \lat^\ell - \vec{t}^\ell$ and if $k$ is even, then $\lat^\ell - k\vec{t}^{\ell} = \lat^\ell$.) 

To bound the number of short vectors, we follow \cite{mazoLatticePointsHighdimensional1990,elkiesPackingDensitiesSuperballs1991,AS18} in using the lattice Theta function 
\[
    \Theta_p(\lat,\tau) := \sum_{\vec{y} \in \lat} e^{-\tau \|\vec{y}\|_p^p}
    \; ,
\]
as in \cref{thm:ball_point_count}. In particular, we note that for any $\tau > 0$, we have 
\begin{equation}
    \label{eq:orthogonal_theta_bound}
    N_p(\lat^\ell, r) < \min_{\tau >0} \, e^{\tau r^p} \Theta_p(\lat^\ell,\tau) = \min_{\tau >0} \, e^{\tau r^p}\Theta_p(\lat,\tau)^\ell
    \; ,
\end{equation}
where the equality uses the fact that 
\[\Theta_p(\lat^\ell,\tau) = \sum_{\vec{y}_1,\ldots, \vec{y}_\ell \in \lat} e^{-\tau (\|\vec{y}_1\|_p^p + \cdots + \|\vec{y}_\ell\|_p^p)} = \Theta_p(\lat,\tau)^\ell
\]
as in \cref{eq:orthogonal-to-power-theta}. In particular, \cref{eq:orthogonal_theta_bound} allows us to obtain bounds on the point-counting function $N_p(\lat^\ell,r)$ of the $(m\ell)$-dimensional lattice $\lat^\ell \subset \R^{m\ell}$, provided that we can bound the Theta function $\Theta_p(\lat,\tau)$ of the $m$-dimensional lattice $\lat \subset \R^m$.   For small $m$ and well-chosen lattices $\lat \subset \R^m$, bounding $\Theta_p(\lat,\tau)$ is doable.

\subsubsection{A construction using the root lattice \texorpdfstring{$D_m$}{Dm}}

Our construction starts with the root lattice
\[D_m:=\{\vec{v}\in\Z^m:\sum_{i=1}^m v_i\equiv 0 \Mod{2} \}
\; ,
\]
and the target $\vec{e}_1 := (1,0,\dots,0)^\top\in\R^m$.
It is easy to see that $D_m$ is a sublattice of $\Z^m$ of index $2$.
Notice also that $D_m$ has $2m$ closest vectors to $\vec{e}_1$ (i.e., the vectors $\vec{e}_1\pm \vec{e}_i$ for $1\le i\le m$), and $\dist_p(\vec{e}_1, D_m) = 1$. Furthermore, $\lambda_1^{(p)}(D_m) = 2^{1/p} > 1$. So, $D_m$ is a promising candidate for our lattice $\lat$ and $\vec{e}_1$ a promising candidate for our target, since there are in fact a relatively large number of vectors at distance $1$ from $\vec{e}_1$ in $D_m$, and only the zero vector has length less than one.

More generally, $D_m^{\ell}$ contains $(2m)^\ell$ closest vectors to $\vec{e}_1^{\ell}$ (i.e., $|\CVP_p(\vec{e}_1^{\ell}, D_m^\ell)| = (2m)^\ell$), with $\dist_p(\vec{e}_1^{\ell}, D_m^\ell) = \ell^{1/p}$. If we could prove this number is larger than $N_p^o(D_m^\ell,\ell^{1/p})$ by an exponential factor, then we can use our techniques to obtain a lattice with exponential kissing number.
However, $N_p^o(D_m^\ell,\ell^{1/p})$ is simply too large for $p = 2+\eps$ and large $\ell$---much larger than the number $(2m)^\ell$ of closest vectors. This arises because of the vectors in $D_m^{\ell}$ with many (but not all) components equal to $\vec0$.

To solve this issue, we apply a simple linear transformation (in fact, a scaled rotation) to the lattice $D_m$ and target $\vec{e}_1$, with the goal of preserving the closest vectors but increasing the $\ell_p$ norms of short vectors. We take the linear transformation to be the Hadamard matrix (defined in \cref{def:Hadamard}). 

Assume $m$ is a power of two, and consider the lattice $\M_m^\ell$ where
\[
    \M_m :=H_mD_m
    := \{ H_m \vec{z} \ : \ \vec{z} \in D_m \} \subset \Z^{m}
    \;, 
\]
and target $\vec t^\ell$ where
\[
     \vec{t}:=H_m\vec{e}_1=(1,\dots,1)^\top\in\Z^m
     \; .
\]
Notice that multiplication by $H_m$ is a scaled isometry for the $\ell_2$ norm (specifically, $\|(H_m/\sqrt{m}) \vec{x}\|_2 = \|\vec{x}\|_2$) so that multiplication by $H_m$ essentially preserves the \emph{Euclidean} geometry of the lattice.%
\footnote{Note that when $m = 1$, $\mathcal{M}_1 = 2\Z$ and $\vec{t} = (1)$, so that the class of lattices $\mathcal{M}_m^\ell$ and targets $\vec{t}^\ell$ strictly generalizes the construction from \cref{sec:Z^n-local-density} using $\Z^n$ and $\vec{t} := (1/2, \ldots, 1/2)^\top$, up to scaling. Rather interestingly, we also have $\mathcal{M}_2 = 2\Z^2$, so that $\mathcal{M}_2^\ell = \mathcal{M}_1^{2\ell} = 2\Z^{2\ell}$. So, both the case $m = 1$ and $m = 2$ yield essentially identical constructions to the construction in \cref{sec:Z^n-local-density}.}
But, the $\ell_p$ norm is \emph{not} preserved by $H_m$.

We will start by proving that multiplying by $H_m$ does in fact preserve the closest vectors in $D_m$ to target $\vec{e}_1$, and in particular, that the number of closest vectors in this construction remains equal to $2m$.

\begin{lemma}[Closest vectors are preserved by Hadamard]\label{lem:closestvectorcount}
For $p>2$, let $\vec t=H_m\vec e_1,\M_m=H_mD_m$, we have $|\CVP_p(\vec{t},\M_m)|=2m$, and $\dist_p(\vec{t},\M_m)=m^{1/p}$.
\end{lemma}
\begin{proof}
Recall $\CVP_p(\vec e_1,D_m)=\{\vec e_1\pm\vec e_i:1\le i\le m\}$. After applying $H_m$, for $\vec{u}\in \CVP_p(\vec e_1,D_m)$, $\|H_m\vec u-\vec t\|_p=\|H_m\vec u-H_m\vec e_1\|_p=m^{1/p}$ by~\cref{fact: Hadamard}, since $\vec u-\vec e_1\in\{\pm\vec e_i:1\le i\le m\}$. Therefore, the vectors in $H_m\CVP_p(\vec e_1,D_m)$ are at distance $m^{1/p}$ to $\vec{t}$, and they are clearly distinct.

It remains to show that these are the closest vectors. For any $\vec{w}\notin H_m\CVP_p(\vec e_1,D_m)$, let $\vec{w}=H_m\vec{z}$ for some $\vec{z}\in D_m$. Then $\vec{z}$ is not a closest vector to $\vec{e}_1$ in $D_m$, hence $\|\vec{z}-\vec{e}_1\|_2> 1$. This implies 
\[
\|\vec w-\vec t\|_p=\|H_m(\vec z-\vec e_1)\|_p\ge m^{1/p}\|\vec z-\vec e_1\|_2>m^{1/p}\ ,
\]
where the first inequality results from~\cref{lem:Hamadard norm}.
This proves $\CVP_p(\vec t,\M_m)=H_m\CVP_p(\vec e_1,D_m)$, so $|\CVP_p(\vec{t},\M_m)|=2m$, and $\dist_p(\vec{t},\M_m)=m^{1/p}$. 
\end{proof}

We have the following construction using $\M_m^\ell,\vec t^\ell$.
\begin{theorem}\label{thm: general G and A}
  Fix any $p>2$, constant $k\ge 0$ and $m=2^k$. Let $\M_m=H_mD_m,\vec t = (1,\dots,1)^\top\in\Z^{m}$. For $\ell>0$, it holds that
$$\frac{G(\M_m^\ell,\vec t^\ell,(m\ell)^{1/p})}{A(\M_m^\ell,\vec t^\ell,(m\ell)^{1/p})}\ge \left(\frac{2m}{\min_{\tau>0} \, e^{\tau m} \cdot \Theta_p(\M_m,\tau)}\right)^\ell \ \text{.}$$  
\end{theorem}
\begin{proof}
    \cref{lem:closestvectorcount} shows $\dist_p(\vec t,\M_m)=m^{1/p}$ and $|\CVP_p(\vec t,\M_m)|=2m$, therefore $\dist_p(\vec t^\ell,\M_m^\ell)=(m\ell)^{1/p}$ and $|\CVP_p(\vec t^\ell,\M_m^\ell)|=(2m)^{\ell}$ (since $\CVP_p(\vec t^\ell,\M_m^\ell)=\CVP_p^{\oplus \ell}(\vec t,\M_m)$). So $$G(\M_m^\ell,\vec t^\ell,(m\ell)^{1/p})=(2m)^{\ell}.$$ 
    Since $\vec{t}^\ell\notin\M_m^\ell$ and $2\vec{t}^\ell\in\M_m^\ell$, $\{ \vec{y} - k\vec{t}^\ell\ : k \in \Z,\ \vec{y} \in \M_m^\ell\}$ is the union of $\{ \vec{y} - k\vec{t}^\ell\ : k \in 2\Z,\ \vec{y} \in \M_m^\ell\} = \M_m^\ell $ and $\{\vec{y} - k\vec{t}^\ell \ : \ k \in 2\Z +1, \vec{y} \in \M_m^\ell \} = \M_m^\ell + \vec{t}^\ell$. Note the second set contributes no vector shorter than $(m\ell)^{1/p}$. Using~\cref{eq:orthogonal_theta_bound}, we get $1\le A(\M_m^\ell,\vec t^\ell,(m\ell)^{1/p})\le \min_{\tau>0} \, e^{\tau m\ell} \cdot \Theta_p(\M_m,\tau)^{\ell}$. Hence the result follows.
\end{proof}

Combining this with \cref{thm:sparsification} gives the following immediate corollary, which shows a relatively simple function that lower bounds the kissing number.

\begin{corollary}\label{cor: alpha to kissing number}
    Fix any $p>2$, constant $k\ge 0$, and $m=2^k$. Define
    \[
    \alpha_{p,m} :=\left(\frac{2m}{\min_{\tau>0} \, e^{\tau m} \cdot \Theta_p(\M_m,\tau)}\right)^{1/m} \ \text{.}
    \]
    For $\ell>0$ and $n:=m\ell$, there exists a full-rank sublattice $\M'\subseteq \M_m^\ell$ and target $\vec t'$, such that $(\M',\vec t')$ is $(p,n,\Omega(\alpha_{p,m}^{n}))$-exact locally dense. In particular, $\kappa_p(\M'+\Z\vec t')\ge \Omega(\alpha_{p,m}^{n})$.

    Therefore, for all $n$, $\kappa_p(n) \geq \Omega(\alpha_{p,m}^{n})$.
\end{corollary}
\begin{proof}
The statement is vacuous when $\alpha_{p,m} \leq 1$, so we may assume $\alpha_{p,m} > 1$.

For $\M_m^\ell,\vec t^\ell$ and radius $(m\ell)^{1/p}$, we have $G/A\ge \alpha_{p,m}^{m\ell}=\alpha^n_{p,m}$ by~\cref{thm: general G and A}. When $\alpha_{p,m}>1$, for sufficiently large $\ell$, this gives $G/A\ge 64$. So, we can apply \cref{thm:sparsification}, which implies the existence of the sublattice $\M'\subseteq \M_m^\ell$ and $\vec t'$ such that $(\M',\vec t')$ is $(p,n,\Omega(\alpha_{p,m}^{n}))$-exact locally dense, and that $\kappa_p(\M'+\Z\vec t')\ge \Omega(\alpha_{p,m}^{n})$.

This implies that $\kappa_p(n) \geq \Omega(\alpha_{p,m}^n)$ for $n$ divisible by $m$. A simple padding argument similar to~\cref{rmk:padding} then shows the result for all $n$.
\end{proof}

\subsubsection{Finishing the proof}\label{sec:finishing}

By \cref{cor: alpha to kissing number}, the kissing number  $\kappa_{p}(n)$ is at least $\Omega(\alpha_{p,m}^n)$ for any power of two $m = 2^k$. Therefore, for any $p$, 
\[
\kappa_p(n) \geq \Omega(\alpha_p^n)
\]
where
\[
    \alpha_p := \max_{k \geq 0} \alpha_{p,2^k}
    \; .
\]
(It is not hard to see that the maximum is achieved. And, the value of $k$ achieving the maximum increases as $p$ approaches $2$ from above.)
But, this is of course not interesting if $\alpha_p \leq 1$. So, it remains to show that $\alpha_p > 1$ so that $\kappa_p(n) \geq 2^{\Omega(n)}$. More generally, we wish to prove the best lower bound on $\alpha_p$ that we can, since better lower bounds imply better kissing numbers.

In fact, for any fixed $p$ and $m$, computing $\alpha_{p,m}$ numerically is not difficult---even for values of $m$ that are quite large. We therefore do exactly that in \cref{fig:high_p_plot}. As we explain in the caption of \cref{fig:high_p_plot}, the computation works by truncating the theta series defining $\alpha_{p,m}$ after a relatively small number of shells and (rigorously) bounding the (small) contribution of later shells. We stopped performing these computations at $m = 2^{17}$, where we found that, e.g., $\alpha_{2.05,m} > 1.000000239$ (and a separate computation shows that $\alpha_{2.05,m} < 1.000000246$).

\begin{figure}[!ht]
    \def\arraystretch{3}
    \centering
    \begin{tabular}{@{}l@{}}
    \includegraphics[width=0.9\textwidth]{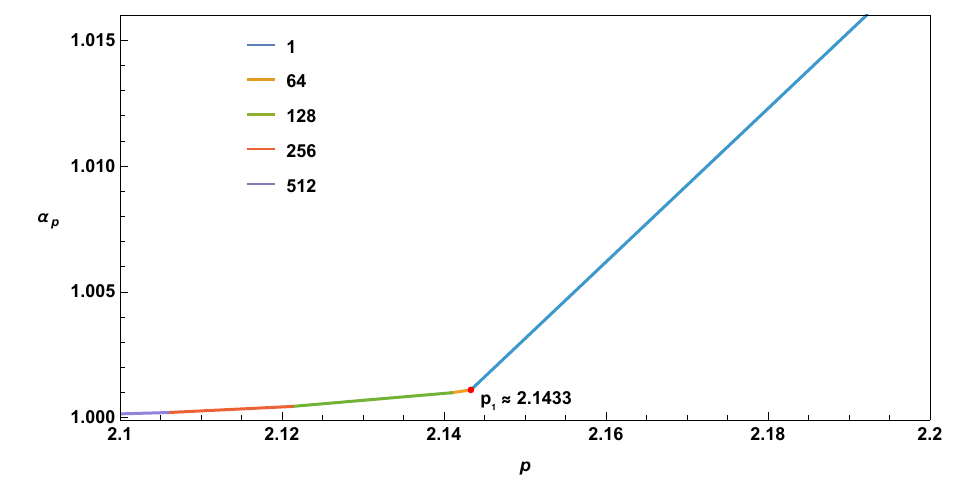} \\
    \includegraphics[width=0.9\textwidth]{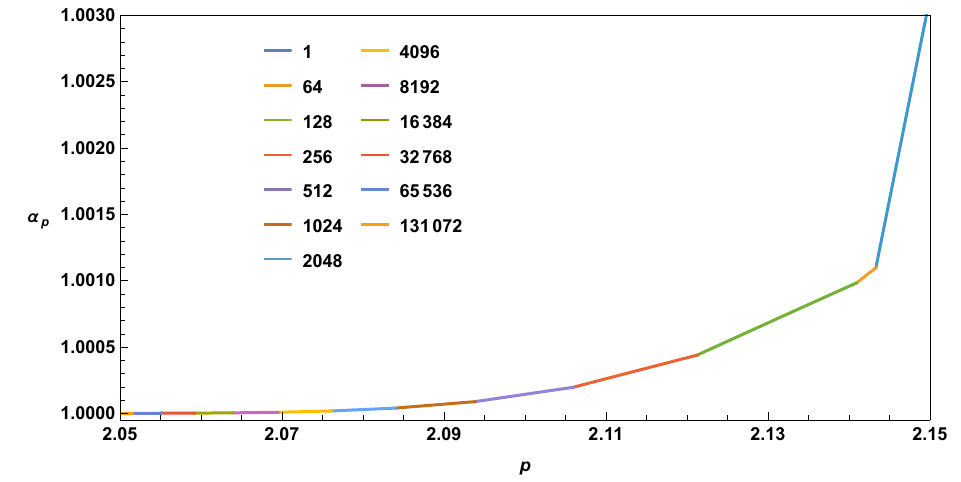}
    \end{tabular}
    \caption[Plot of $\alpha_p$]{Numerical estimates of $\alpha_{p}$ for different values of $p$. The top plot shows the maximum value of $\alpha_p$ over all powers of two for $2.1 \leq p \leq 2.2$. The colors represent the value of $m$ that achieved the optimal value $\alpha_p = \max_m \alpha_{p,m}$.
    The bottom plot shows the values $2.05 \leq p \leq 2.15$, where much larger values of $m$ become relevant. 
    
    \protect \par Formally, these values were computed by computing the exact contribution to the sum $\Theta_p(\mathcal{M}_{m},\tau)$ of vectors $\vec{y} \in \mathcal{M}_m$ with $\|\vec{y}\|_p^p/m$ equal to $0$, $2^{p-1}$, $2^{p}$, $2^{p-1} + 2^{2p-3}$, $2^{2p-2}$, and $2^{p-1} + 2^{2p-2}$ (i.e., the first five shells together with zero) and using \cref{lem:Hamadard norm} together with an exact expression for $\Theta_2(\mathcal{M}_{m},\tau)$ in terms of Jacobi theta functions to bound the contribution of lower-order terms. (For $m = 1$, we used a simpler truncated sum up to the tenth shell, with the same rigorous bound on the error.) So, this numerical estimate is by design a (very tight) lower bound on $\alpha_{p,m}$.}
    \label{fig:high_p_plot}
\end{figure}

It remains to rigorously show that $\alpha_p := \max \alpha_{p,m}$ is always strictly greater than $1$ for all $p > 2$. More generally, we prove a relatively simple analytic bound on $\alpha_{p,m}$. This bound starts with the following lemma.\footnote{We note that the theta function $\Theta_2(D_m,\sigma)$ satisfies the well known (and easy to prove) identity
\[
    \Theta_2(D_m,\sigma) = \frac{\theta_3(\sigma)^m + \theta_4(\sigma)^m}{2}
    \; ,
\]
where $\theta_3(\sigma) := \sum_{z \in \Z} e^{-\sigma z^2}$ and $\theta_4(\sigma) := \sum_{z \in \Z}(-1)^z e^{-\sigma z^2}$ are the Jacobi theta functions. The Jacobi theta functions are well understood.}

\begin{lemma}\label{lem: directsumshortbound}
For any $p > 2$ and $\tau>0$,
\[
\Theta_p(\M_m,\tau)\le\Theta_2(\M_m,\tau/\mu)=\Theta_2(D_m,\tau m/\mu)
\; ,
\]
where $\mu := 2^{1 - p/2} < 1$.
\end{lemma}
\begin{proof}
For all $\vec{v}\in D_m$, if $\vec{v}\ne \bm{0}$, then $\|\vec{v}\|_2\ge \sqrt{2}$, so $\|H_m\vec v\|_2\ge \sqrt {2m}$.
Therefore, \cref{fact: norm conversion} gives
\[
\|H_m\vec{v}\|_p^p\ge (m^{1/p-1/2}\|H_m\vec{v}\|_2)^p=m^{1-p/2} \|H_m \vec{v}\|_2^{p-2} \|H_m\vec{v}\|_2^2\ge m^{1-p/2}(\sqrt{2m})^{p-2}\|H_m\vec{v}\|_2^2=\|H_m\vec{v}\|_2^2/\mu.
\]
From this, it follows that
\[
\Theta_p(\M_m,\tau)=\sum_{\vec{v}\in D_m}e^{-\tau \|H_m\vec{v}\|_p^p}\le \sum_{\vec{v}\in D_m}e^{-\tau \|H_m\vec{v}\|_2^2/\mu}=\Theta_2(\M_m,\tau /\mu).
\]
Finally, the equality follows from the fact that $H_m$ is a scaled rotation. In particular,
\[
\Theta_2(\M_m,\tau/\mu)=\sum_{\vec{v}\in D_m}e^{-\tau \|H_m\vec{v}\|_2^2/\mu}=\sum_{\vec{v}\in D_m}e^{-\tau m \|\vec{v}\|_2^2/\mu}=\Theta_2(D_m,\tau m/\mu). \qedhere
\]
\end{proof}

\begin{lemma}\label{lem:chosen tau bound}
   For any $p > 2$ and $m \geq 2$ a power of two, we have
   \[
   \alpha_{p,m}^m :=\frac{2m}{\min_{\tau>0} \, e^{\tau m}\Theta_p(\M_m,\tau)} \ge \frac{2m^{1-\mu}}{9} \ \text{,}
   \]
   where $\mu := 2^{1 - p/2} < 1$.
\end{lemma}
\begin{proof}
\cref{lem: directsumshortbound} gives $\Theta_p(\M_m,\tau)\le \Theta_2(D_m,\tau m/\mu)$, and we trivially have  that $\Theta_2(D_m,\tau m/\mu)\le \Theta_2(\Z^m,\tau m/\mu)$.
Let $\tau :=\mu\ln(m)/m$. Then,
\[
\Theta_p(\M_m,\tau)\le \sum_{\vec{v}\in\Z^m}e^{-\tau m\|\vec{v}\|_2^2/\mu}\le \sum_{\vec{v}\in\Z^m}m^{-\|\vec{v}\|_2^2}=\left(\sum_{z\in\Z}m^{-z^2}\right)^m \ \text{,}
\]
and
\[
\sum_{z\in\Z}m^{-z^2}=1+2\sum_{z\in\Z_{>0}}m^{-z^2}\le 1+2\sum_{z\in\Z_{>0}}m^{-z}=1+\frac{2m^{-1}}{1-m^{-1}}=1+\frac{2}{m-1} \ \text{.}
\]

Differentiation shows that $\left(1+\frac{2}{m-1}\right)^m$ is strictly decreasing for $m\ge 2$, and so
\[
\Theta_p(\M_m,\tau)\le \left(1+\frac{2}{m-1}\right)^m \leq \left(1+\frac{2}{2-1}\right)^2 = 9 \ \text{.}
\]
We have that $e^{\tau m}=m^{\mu}$, and so this gives
\[
\alpha_{p,m}^m\ge\frac{2m}{9\cdot m^\mu}=\frac{2m^{1-\mu}}{9} \ \text{,}
\]
as needed.
\end{proof}

Finally, plugging in a concrete value of $m$ as a function of $p$ gives the following result.

\begin{corollary}[Exponential kissing number for $p>2$]\label{thm:last-one}
For any constant $p > 2$, 
\[
    \kappa_p(n) \geq \Omega(\alpha_p^n)
    \geq 2^{\Omega(n)}
    \; .
\]
\end{corollary}
\begin{proof}
Let $\mu := 2^{1-p/2} < 1$.
Let $\widehat{m} := e (9/2)^{1/(1-\mu)}$. A quick check shows that the function $f_p(m) := (2m^{1-\mu}/9)^{1/m}$ arising in \cref{lem:chosen tau bound} is decreasing for $m > \widehat{m}$. So, let $m$ be the first power of two larger than $\widehat{m}$, so that in particular $f_p(m) \geq f_p(2\widehat{m})$. It follows from \cref{lem:chosen tau bound} that
\[
    \alpha_p \geq \alpha_{p,m} \geq f_p(m) \geq f_p(2\widehat{m}) = (2e)^{(1-\mu)/(2 \widehat{m})} > 1
    \; ,
\]
and the result follows from \cref{cor: alpha to kissing number}.
\end{proof}

\newcommand{\etalchar}[1]{$^{#1}$}

\appendix
\crefalias{section}{appendix}
\crefname{appendix}{Appendix}{Appendices}

\section{Proof of sparsification results}\label{app:A}

We will begin by proving some technical lemmas used in the proof of~\cref{lem:sparsification-25}. 
\begin{lemma}\label{lem:order-of-t}
    Let $p\ge 1$, $\lat$ be a lattice of rank $n$ with basis $\basis$, $\vec{t}\in\spn(\lat)$, and $r > 0$. Let 
    \[A:=|\{\vec{y}-k\vec{t}:\vec{y}\in\lat,k\in\Z,\|\vec{y}-k\vec{t}\|_p<r\}|
    \; .
    \]
    If $A$ is finite, then there exists an integer $z>0$, such that $z\vec{t}\in\lat$. Furthermore, $\lat+\Z\vec t$ is a lattice.
\end{lemma}
\begin{proof}
Let $\lat=\lat(\basis)$, $S=[0,1)^n\subseteq \R^n$, define $\mathcal{P}(\lat):=\basis S$. For any $\vec{x}\in \mathrm{span}(\lat)$, it is easy to see that there exist unique $\vec{v}\in\lat,\vec{p}\in\mathcal{P}(\lat)$, such that $\vec{x}=\vec{v}+\vec{p}$. We write $\vec{p}=\vec{x}\mod \lat$. 

Now, if for all integer $z\ge 1$ we have $z\vec{t}\notin\lat$, let $Z:=\{z\vec{t}\mod\lat:z\in\Z_{\ge 1}\}$. Then $\bm{0}\notin Z$, and for $1\le k<l$ we must have $k\vec{t}\mod\lat\ne l\vec{t}\mod\lat$, otherwise $(l-k)\vec{t}\in\lat$. Then $Z\subset \overline{\mathcal{P}(\lat)}$ is an infinite set inside a compact set, thus for all $\eps>0$, there exist $1\le h<j$ such that $\|(j\vec{t}\mod\lat)-(h\vec{t}\mod\lat)\|_p<\eps$. Write $j\vec{t}=\vec{v}_j+\vec{p}_j,h\vec{t}=\vec{v}_h+\vec{p}_h$ where $\vec{v}_j,\vec{v}_h\in\lat,\vec{p}_j,\vec{p}_h\in \mathcal{P}(\lat)$. Since $(j-h)\vec{t}\notin\lat$, $0<\|(j-h)\vec{t}-(\vec{v}_j-\vec{v}_h)\|_p=\|\vec{p}_j-\vec{p}_h\|_p<\eps$. This contradicts $A$ being finite. So $z\vec t\in\lat$ for some $z\in\Z$.

Then $\lat+\Z\vec t=\cup_{0\le i<z}(\lat+i\vec t)$ is discrete and additive, and hence it is a lattice.
\end{proof}
The following concentration bound will be used in our sparsification.

\begin{lemma}[Chebyshev concentration]\label{lem:chebyshev}
Let $q$ be a prime, $\vec h_1,\dots,\vec h_G\in\Z_q^n\setminus\{\bm0\}$ be pairwise linearly independent, i.e.,
$\vec h_i\not\equiv c\vec h_j\pmod q$ for all $i\neq j$ and $c\in\Z_q$. Let $\vec a$ be uniform over $\Z_q^n$. Set
\[
X:=\big|\{i\in[G]:\langle\vec h_i,\vec a\rangle\equiv0\pmod q\}\big|.
\]
Then $\E[X]=\frac{G}{q}$ and $\Var[X]= \frac{G}{q}(1-\frac1q)$. Consequently, for every $\delta>0$,
\[
\Pr[X\le(1-\delta)G/q]
\le\frac{1-1/q}{\delta^2G/q}<\frac{q}{\delta^2G}.
\]
\end{lemma}
\begin{proof}
Let $X_i:=\bm{1}\{\langle\vec h_i,\vec a\rangle\equiv0\pmod q\}$.
Then $\E[X_i]=\Pr[\langle\vec h_i,\vec a\rangle\equiv0\pmod q]=\frac1q$, so $\E[X]=G/q$. Pairwise linear independence gives $\E[X_iX_j]=\frac1{q^2}$ for $i\ne j$. Hence $\Var[X]= \frac{G}{q}(1-\frac1q)$. The result then follows from Chebyshev's inequality.
\end{proof}

The next two technical lemmas imply a large number of short vectors given some short vectors with certain properties. 

\begin{lemma}\label{lem: short vec not in qL}
For a lattice $\lat$, $p\ge 1$ and $r>0$, if $\bm{0}\ne \vec{y}\in q\lat$ and $\|\vec{y}\|_p\le r$, then we have 
$$N^o_p(\lat,r)\ge 2q-1.$$ 
\end{lemma}
\begin{proof}
$\bm{0},\pm\vec{y}/q,\pm 2\vec{y}/q,\dots,\pm(q-1)\vec{y}/q$ are distinct vectors in $\lat$ with norm less than $r$, which proves the lemma.
\end{proof}

\begin{lemma}\label{lem: equal norm not independent}
For a lattice $\lat$, $p\ge 1$, $r>0$, and lattice vectors $\vec{v}, \vec{u} \in \lat$ with $\|\vec{v}\|_p = \|\vec{u}\|_p =r$ and $\vec{v} \neq \pm \vec{u}$. Let $q\ge 3$ be an odd integer. If $\vec{v}\equiv k\vec{u}\pmod{q\lat}$ for some integer $k$, then 
$$N^o_p(\lat,r)\ge q.$$ 
\end{lemma}
\begin{proof}Suppose $\vec{v}=k\vec{u}+q\vec{w}$ for some $k\in\Z,\vec{w}\in \lat$. For each $a\in \{-(q-1)/2,\dots,(q-1)/2\}$, choose $b_a\in\{-(q-1)/2,\dots,(q-1)/2\}$ such that $b_a=-ak+ c_aq$ for some integer $c_a$. Set 
$$
    \vec{s}_a=\frac{a\vec{v}+b_a\vec{u}}{q}=\frac{ak\vec{u}+aq\vec{w}-ak\vec{u}+c_aq\vec{u}}{q}=a\vec{w}+c_a\vec{u}\in \lat.
$$
But we have $\|\vec{s}_a\|_p\le \frac1q (|a|\|\vec{v}\|_p+|b_a|\|\vec{u}\|_p)=\frac{|a|+|b_a|}{q}r<r$. Now for $-(q-1)/2\le a<a'\le (q-1)/2$, if $\vec{s}_a=\vec{s}_{a'}$, this implies $a\vec{v}+b_a\vec{u}=a'\vec{v}+b_{a'}\vec{u}$. But $0<\|\vec{v}\|_p=\|\vec{u}\|_p$ and $\vec{v}\ne\pm\vec{u}$, so they are linearly independent, in contradiction to $a<a'$. This proves $\{\vec{s}_a\}_{-(q-1)/2\le a\le (q-1)/2}$ contains $q$ distinct vectors in $\lat$ with norm less than $r$.
\end{proof}

Now, we are ready to prove the main sparsification results.
\begingroup
\renewcommand{\thetheorem}{\getrefnumber{lem:sparsification-25}}
\prop*
\endgroup
\addtocounter{theorem}{-1}
\begin{proof}
Note that $q\lat\subseteq \lat' \subseteq \lat$, so that $\lat'$ has rank $n$ as claimed. \cref{lem:order-of-t} implies $s>1$ is well-defined, and $M:=\lat+\Z\vec t$ is a lattice. Observe that $s\vec t' = \basis (\vec d-s\vec z)$, and $\langle \vec d-s\vec z,\vec a\rangle = \langle \vec d-ss^*\vec d,\vec a\rangle\equiv 0\pmod q$. Hence $s\vec t'\in\lat'$, and $\lat'+\Z\vec t'$ is also a lattice.

\paragraph{Proof of~\cref{item-one-BPT25}.} Let $A_2:=\{\vec{y}-k\vec{t}:\vec{y}\in\lat,k\in\Z,0<\|\vec{y}-k\vec{t}\|_p<r\}$, then $A=|A_2|+1$. For any $\vec w\in A_2$, write $\vec w=\vec y-k\vec t,\ \vec y=\basis\vec x$ where $\vec x\in\Z^n$. We can assume $0\le k<s$ since $s\vec t\in\lat$. Note $\vec w = (\vec y-k\basis\vec z)-k\vec t'$. We next prove $\vec w\in\lat'+\Z\vec t'\iff \vec y-k\basis\vec z\in\lat'$. The converse direction is trivial. Suppose $\vec w=\vec y'-e\vec t'$ for some $\vec y'\in\lat'$, $s\vec t'\in\lat'$ implies we can assume $0\le e<s$. Then we have $\vec y-k\basis\vec z-\vec y'=(k-e)\vec t'=(k-e)(\vec t-\basis \vec z)$. Since the left hand side is in $\lat$ and $-s<(k-e)<s$, we get $k=e$, and hence $\vec y-k\basis\vec z=\vec y'\in\lat'$. We have
\begin{equation}
    \vec w\in\lat'+\Z\vec t'\iff \vec y-k\basis\vec z\in\lat'\iff \langle\vec x-k\vec z,\vec a\rangle\equiv 0\pmod q. \label{eq:bad w condition}
\end{equation}
Recall $s\vec t=\basis\vec d$, and~\cref{eq:bad w condition} is equivalent to $ \langle \vec x-ks^*\vec d,\vec a\rangle\equiv 0\pmod q$. Since $q\nmid s$, multiplying by $s$ gives equivalence to $\langle s\vec x-k\vec d,\vec a\rangle\equiv 0\pmod q$. 

We now prove $s\vec x-k\vec d\not\equiv0\pmod q$. Otherwise $s\vec w=\basis (s\vec x-k\vec d)\in q\lat\subseteq qM$. Since $\vec w\in M$ and $q\nmid s$, we have $\vec w\in qM$. But $0<\|\vec w\|_p<r$, \cref{lem: short vec not in qL} implies $A\ge 2q-1$, contradicting the assumption that $q>A$. Therefore,
\[
\Pr[\vec w\in\lat'+\Z\vec t']=\Pr[\langle s\vec x-k\vec d,\vec a\rangle\equiv 0\pmod q]=\frac1q.
\]

Since $\lat'+\Z\vec t'\subseteq \lat+\Z\vec t$, a union bound over $\vec w\in A_2$ shows
\[
\Pr[\lambda_1^{(p)}(\lat'+\Z\vec t')<r]\le \frac{A-1}{q}.
\]

\paragraph{Proof of~\cref{item-two-BPT25}.} Write $\{\vec v\in \lat:\|\vec v-\vec t\|_p=r\}=\{\vec v_1,\dots,\vec v_G\},\ \ \vec v_i=\basis\vec x_i$ for $1\le i\le G$, and define $\vec u_i=\vec v_i-\vec t$. Let $X :=N_p(\lat'-\vec t',r)$. Note that 
\[\{ \vec{v} \in \lat' \ : \ \|\vec{v}-\vec{t}'\|_p = r\}=(\{\vec v\in \lat:\|\vec v-\vec t\|_p=r\}-\basis \vec z)\cap \lat'
\; ,
\]so that $X \geq \sum_{i=1}^G \bm{1}\{\vec v_i-\basis\vec z\in\lat'\}$.

Define $\vec h_i=s\vec x_i-\vec d$, for $1\le i\le G$. We have
\begin{equation}
    \vec v_i-\basis\vec z\in\lat'\iff \langle\vec x_i-s^*\vec d,\vec a\rangle\equiv 0\pmod q\iff \langle\vec h_i,\vec a\rangle\equiv 0\pmod q. \label{eq:condition of bi}
\end{equation}

We next show $\vec h_i\not\equiv 0\pmod q$. Indeed, otherwise $\basis\vec h_i=s\vec v_i-s\vec t=s\vec u_i\in q\lat\subseteq qM$. Since $\vec u_i\in M$ and $q\nmid s$, we have $\vec u_i\in qM$, and~\cref{lem: short vec not in qL} gives a contradiction.

Suppose for $i\ne j$, $\vec h_i\equiv k\vec h_j\pmod q$ for some $k\in\Z_q$. Then $s\vec{v}_i-s\vec t-ks(\vec v_j-\vec t)=s(\vec u_i-k\vec u_j)\in qM$, so $\vec u_i-k\vec u_j\in qM$. So $\vec u_i\equiv k\vec u_j\pmod {qM}$, since they have equal norm $r$, \cref{lem: equal norm not independent} implies $A\ge q$ unless $\vec u_i=\pm\vec u_j$. For $i\ne j$, this happens only if $\vec u_i=-\vec u_j$, so $\vec v_i-\vec t=\vec t-\vec v_j$. Then $2\vec{t}\in\lat$ and $s=2$.

If $s>2$, vectors $\vec h_1,\dots,\vec h_G$ are pairwise linearly independent modulo $q$. \cref{lem:chebyshev} gives 
\[
\Pr[X\le (1-\delta)G/q]< \frac{q}{\delta^2 G}.
\]
Else when $s=2$, two vectors $\vec h_i=-\vec h_j$ always survive simultaneously. We group $\vec h_1,\dots,\vec h_G$ into $G/2$ such pairs, the same lemma implies $\Pr[X\le (1-\delta)G/q]< \frac{2q}{\delta^2 G}.$
\end{proof}

\begin{remark}
We assume that $q\nmid s$ in~\cref{lem:sparsification-25} for convenience. However, we note that a different construction works if $q\mid s$. In that case, one can sample $\vec{z} \sim \Z_q^n$ and $\vec a$ uniformly from $\{\vec x\in\Z^n_q:\langle \vec x,\vec d\rangle\equiv 0\pmod q\}$. This yields similar results for minimum distance and number of close vectors. However, in our applications (\cref{thm:large p,cor: alpha to kissing number}), $s=2$ and $q$ is a large prime. Therefore, we prove~\cref{lem:sparsification-25} only for $q\nmid s$.

Furthermore, one can also think of our choice of $\vec a,\vec z$ as following~\cite[Proposition~2.5]{conf/innovations/BennettPT22}, but conditioned on $s\vec t'\in\lat'$. This is equivalent to choosing a fixed $\vec z=s^*\vec d$ as in~\cref{lem:sparsification-25}.
\end{remark}

\section{Building exact locally dense lattices from lattices with large kissing number}
\label{app:going_backwards}

In this section, we show a simple converse to \cref{thm:kannans-embedding-formal}, that is, we show that a lattice with large kissing number implies an exact locally dense lattice. We note that a similar result was proven by Aggarwal and Stephens-Davidowitz. Namely, they showed how to build a certain kind of locally dense lattice from a lattice with large kissing number. (See \cite[Theorem 5.9]{AS18}, \cite[Lemma 3.6]{conf/innovations/BennettPT22}, and the discussion above.) Our result is in some sense much simpler, because we are interested in \emph{exact} locally dense lattices. On a technical level, they required $\lambda_1^{(p)}(\lat') \geq (1+\eps) \dist_p(\vec{t}',\lat')$ for some positive constant $\eps > 0$, whereas we are happy with $\eps = 0$. On the other hand, we require many lattice vectors at distance exactly $\dist_p(\vec{t}',\lat')$, while \cite{AS18,conf/innovations/BennettPT22} only required many lattice vectors within distance $(1-\eps) \lambda_1^{(p)}(\lat')$. (However, in our context, this constraint is remarkably easy to satisfy.)

\begin{lemma} \label{lem:new-KN-to-LDL}
For any $p \geq 1$ and lattice $\lat$ of rank $n$, there exists an index-two sublattice $\lat' \subsetneq \lat$ of $\lat$ and vector $\vec{t} \in \lat \setminus \lat'$ such that $(\lat', \vec{t})$ is a $(p, n, N)$-exact locally dense lattice for $N > \kappa_p(\lat)/2$.
In particular, for every $p \geq 1$ and $n \in \Z^+$ there exists a $(p, n, N)$-exact locally dense lattice for $N > \kappa_p(n)/2$.
\end{lemma}

\begin{proof}
Let $S := \set{\vec{v} \in \lat : \norm{\vec{v}}_p = \lambda_1^{(p)}(\lat)}$ be the set of the shortest non-zero vectors of $\lat$.

We note that it suffices to prove that there exist an index-two sublattice $\lat' \subsetneq \lat$ of $\lat$ and a vector $\vec{t} \in \lat \setminus \lat'$ such that the coset $\lat' - \vec{t} = \lat \setminus \lat'$ satisfies $\card{S \cap (\lat' - \vec{t})} > \kappa_p(\lat)/2$. 
Indeed, notice that $\card{\CVP_p(\vec{t}, \lat')} = \card{S \cap (\lat' - \vec{t})}$ as long as $S \cap (\lat' - \vec{t})$ is non-empty. So, if $\card{S \cap (\lat' - \vec{t})} > \kappa_p(\lat)/2$ then $(\lat', \vec{t})$ is a $(p, n, N)$-exact locally dense lattice for $N>\kappa_p(\lat)/2$.

To that end, let $\basis$ be a basis of $\lat$ and sample $\vec{a} \sim \bit^n \setminus \set{\vec{0}}$ uniformly at random. Let
\[
\lat' := \set{\basis \vec{x} : \vec{x} \in \Z^n, \iprod{\vec{x}, \vec{a}} \equiv 0 \Mod{2}} \ \text{,}
\]
and let $\vec{t} = \basis \vec{z} \in \lat \setminus \lat'$. (Notice that such a vector $\vec{t}$ exists since $\vec{a} \not\equiv \vec{0} \Mod{2}$.)

Let $\vec{v} = \basis\vec{y} \in S$ and notice that $\vec{y} \not\equiv \vec{0} \Mod{2}$ (since otherwise $\vec{v} = \vec{0}$ or $\vec{v}/2 \in \lat$ and $\vec{v}/2$ is shorter than $\vec{v}$). We then have that
   \[
    \Pr_{\vec{a}}[\vec{v} \in \lat'-\vec t ] = \Pr_{\vec{a}}[\langle\vec y,\vec a\rangle\equiv 1 \Mod{2}]> 1/2\ \text{.}
    \]
So, by linearity of expectation, $\E_{\vec a}[\ \card{S \cap (\lat' - \vec{t})}\ ]>|S|/2=\kappa_p(\lat)/2$, hence there exists a pair $(\lat',\vec t)$ satisfying $\card{S \cap (\lat' - \vec{t})} >\kappa_p(\lat)/2$, as needed. 
\end{proof}

We then obtain a proof of \cref{thm:LDL-KN-equivalence}, which gives a near equivalence between lattices with large kissing number and exact locally dense lattices.

\begin{proof}[Proof of \cref{thm:LDL-KN-equivalence}]
Combine \cref{thm:kannans-embedding-formal,lem:new-KN-to-LDL}.
\end{proof}

\section{Refinement of Wan's theorem}
\label{app:Wan_stuff}

In this section, we show how to extend the proof of~\cite[Theorem~6.3]{wan2026nphardnesssvpeuclideanspace} to $k=\Theta(q^{1/4})$.

\corRS*
\begin{proof}\label{app}
 Observe that although the parameters we choose are not covered directly by the statement of \cite[Theorem~6.3]{wan2026nphardnesssvpeuclideanspace}, the same proof applies. Indeed, by \cite[(6.3.2)]{wan2026nphardnesssvpeuclideanspace},
\[
     |S_2(\vec y,2k)|\cdot (2k)!>q^{(k+2)/2}\binom{2k}{2}
    \big(q^{(k-2)/2}-(2k)^{2k}\big).
\]
Moreover,
\[
    \frac{(2k)^{2k}}{q^{(k-2)/2}}=    q\left(\frac{(2k)^2}{q^{1/2}}\right)^k
    \le q\left(\frac{\alpha^2q^{1/2}}{q^{1/2}}\right)^k= q\alpha^{2k}=o(1),
\]
since $0<\alpha<1$. Hence,
\[
    q^{(k-2)/2}-(2k)^{2k}=(1-o(1))q^{(k-2)/2},
\]
and therefore $|S_2(\vec y,2k)|\cdot (2k)!\ge(1-o(1))\binom{2k}{2}q^k$. Using $(2k)!\le (2k)^{2k}\le \alpha^{2k}q^{k/2}\le q^{k/2}$, we obtain
\[
    |S_2(\vec y,2k)|\ge q^{k/2}=2^{\Omega(q^{1/4}\log q)}.
\]
Thus the proof of \cite[Theorem~6.3]{wan2026nphardnesssvpeuclideanspace} extends to the endpoint \(k=\Theta(q^{1/4})\).
\end{proof}

\section{A brief note on why a different construction in \texorpdfstring{\cite{AS18}}{[AS18]} did not work for us}
\label{app:stuff_that_doesn't_work}
Here, we briefly discuss a different way that one might try to build exact locally dense lattices (and thus lattices with large kissing number) in our context. We explain that we were unable to make this idea work, though we do not know how to rule it out. 

Specifically, we note that \cite{AS18} actually has two slightly different constructions of locally dense lattices from $\Z^n$. The first is the one that we discuss in \cref{sec:Z^n-local-density} in which the target vector is simply $\vec{t} := (1/2,\ldots, 1/2)^\top \in \R^n$. However, as \cite{AS18} noted, this construction fails for $p \leq p_0 \approx 2.1397$ specifically because for such $p$, $N_p(\Z^n, n^{1/p}/2) > 2^{n-o(n)}$.

\cite{AS18} therefore presented a different construction of a locally dense lattice (see \cite[Lemma 5.4]{AS18}) that still uses $\Z^n$ as the lattice but takes the target vector $\vec{t}$  to be $(\eps,\ldots, \eps)^\top \in \R^n$ for some small constant $\eps > 0$ (depending on $p$). For all $p > 2$, they show that with this construction there exist constants $0 < C_p < C_p'$ such that $N_p(\Z^n-\vec{t},C_p n^{1/p}) > 2^{\Omega_p(n)} \cdot N_p(\Z^n, C_p' n^{1/p})$. In other words, there are exponentially many more lattice points in a ball centered around $\vec{t}$ than there are in a ball around the origin (even if the radius of the ball around the origin is slightly larger than the ball around $\vec{t}$). (See also~\cite{aggarwalMindGapNot2026}, which has a similar construction working with $\vec{t} = (1/2,\dots, 1/2)^\top$ with a varying radius.)

This construction is sufficient for one of their purposes (proving Gap-ETH hardness of $\SVP_p$), but we do not see how to use it to build an \emph{exact} locally dense lattice. It is natural to try to sparsify such a construction to obtain a locally dense lattice, like we do in \cref{thm:sparsification} (and as \cite{AS18} do implicitly). However, it seems that for such a technique to yield an \emph{exact} locally dense lattice, it is necessary for there to be some fixed radius $r$ such that the number of points in $\lat$ at distance \emph{exactly} $r$ from $\vec{t}$ is exponentially larger than the number of \emph{closer} lattice points to $\vec{t}$ (in addition to the short lattice vectors and the close vectors to $k\vec{t}$). And, we simply do not see how to get around this obstruction using this construction. 

In contrast, the construction using $\vec{t} = (1/2,\ldots, 1/2)^\top$ presented in \cref{sec:Z^n-local-density} and our generalization using $\mathcal{M}_m$ in \cref{sec:Dm_rotation_thing} work with a radius $r$ that is simply equal to $\dist_p(\vec{t},\lat)$. In this case, one of course does not need to worry about any lattice vectors with distance from $\vec{t}$ strictly less than $r$, since they do not exist. (One certainly still needs to worry about short vectors and close vectors to $k\vec{t}$.)
\end{document}